\documentclass{daj}

\usepackage{amsthm,amsmath,amssymb,booktabs,xcolor,graphicx}
\usepackage{bbm}
\usepackage{listings}
\usepackage{xcolor}
\usepackage{appendix}
\usepackage{enumerate}
\usepackage[shortlabels]{enumitem}
\usepackage{hyperref}
\usepackage[noabbrev,capitalize]{cleveref}
\crefname{equation}{}{} %remove ``Equation''
\numberwithin{equation}{section}
\newtheorem{theorem}{Theorem}[section]
\newtheorem{proposition}[theorem]{Proposition}
\newtheorem{lemma}[theorem]{Lemma}
\newtheorem{claim}[theorem]{Claim}
\crefname{claim}{Claim}{Claims}

\newtheorem{corollary}[theorem]{Corollary}

\newtheorem*{question*}{Question}

\theoremstyle{definition}
\newtheorem{definition}[theorem]{Definition}

\newtheorem*{definition*}{Definition}

\theoremstyle{remark}
\newtheorem*{remark}{Remark}

\newcommand{\abs}[1]{\left\lvert#1\right\rvert}
\newcommand{\norm}[1]{\left\lVert#1\right\rVert}

\newcommand{\floor}[1]{\left\lfloor #1 \right\rfloor}

\newcommand{\paren}[1]{\left( #1 \right)}

\newcommand{\set}[1]{\left\{ #1 \right\}}

\newcommand{\one}{\mathbbm{1}}

\newcommand{\mbm}{\mathbbm}
\newcommand{\mc}{\mathcal}

\newcommand{\CC}{\mathbb{C}}
\newcommand{\EE}{\mathbb{E}}
\newcommand{\FF}{\mathbb{F}}
\newcommand{\NN}{\mathbb{N}}
\newcommand{\PP}{\mathbb{P}}
\newcommand{\QQ}{\mathbb{Q}}
\newcommand{\RR}{\mathbb{R}}

\newcommand{\ZZ}{\mathbb{Z}}

\newcommand{\pdd}{\mathsf{pdd}}
\usepackage{etoolbox}
\usepackage{appendix}
\usepackage[capitalize]{cleveref}

\AtBeginEnvironment{appendices}{\crefalias{section}{appendix}}

\dajAUTHORdetails{%
  title = {Patterns without a Popular Difference}, %% please capitalize all significant words
  author = {Ashwin Sah, Mehtaab Sawhney, and Yufei Zhao},
  plaintextauthor = {Ashwin Sah, Mehtaab Sawhney, Yufei Zhao},
  keywords = {Szemer\'edi's theorem, popular difference},
}   %%% END \dajAUTHORdetails

\dajEDITORdetails{%
   year={2021},
   number={8},
   received={24 April 2020},   % received date: example: 7 January 2017
   published={30 July 2021},  % published date
   doi={10.19086/da.25317},       % XXX = number of paper, e.g. da006 for paper#6
}   %%% END \dajEDITORdetails

\begin{document}

\begin{frontmatter}[classification=text]
%% EDITOR: this will force the keywords to appear right after the Abstract.
%%   If the abstract is too long and would force the keywords off the
%%   front page, please comment out % [classification=text] above
%%   This way the keywords will be floated on the bottom of the first page
%%   even though the Abstract spills over to the next page.

%%% AUTHOR: Title goes here.  This line is optional.  You must use it
%%   if title has footnote attached or requires nontrivial typesetting,
%%   e.g., inclusion of linebreaks to force nice layout.
\title{Patterns without a Popular Difference}%\footnote{This is a footnote to the title}} %% please capitalize all significant words

%%% AUTHOR:
%%% List all authors. If you wish, place grant acknowledgements in \thanks.
%%% In brackets include a short tag for each author.
\author[Sah]{Ashwin Sah}%\thanks{Supported by...}}
\author[Sawhney]{Mehtaab Sawhney}%\thanks{Supported by...}}
\author[Zhao]{Yufei Zhao\thanks{Supported by NSF Award DMS-1764176, a Sloan Research Fellowship, and the MIT Solomon Buchsbaum Fund}}

%%% AUTHOR: Abstract goes here
\begin{abstract}
Which finite sets $P \subseteq \ZZ^r$ with $|P| \ge 3$ have the following property:
for every $A \subseteq [N]^r$, there is some nonzero integer $d$ such that $A$ contains $(\alpha^{|P|} - o(1))N^r$ translates of $d \cdot P = \{d p : p \in P\}$, where $\alpha = |A|/N^r$?

Green showed that all 3-point $P \subseteq \ZZ$ have the above property. 
Green and Tao showed that 4-point sets of the form $P = \{a, a+b, a+c, a+b+c\} \subseteq \ZZ$ also have the property. 
We show that no other sets have the above property. 
Furthermore, for various $P$, we provide new upper bounds on the number of translates of $d \cdot P$ that one can guarantee to find.
\end{abstract}
\end{frontmatter}

%%% AUTHOR: body of paper starts here
\section{Introduction}

Green~\cite{Green05} proved a strengthening of Roth's theorem on 3-term arithmetic progressions, showing that for every $A \subseteq [N] := \{1, \dots, N\}$, there exists some ``popular common difference'' $d \ne 0$ such that
\begin{equation} \label{eq:3-ap-popular}
\abs{\set{t : t, t+d, t+2d \in A}} \ge (\alpha^3 - o(1)) N,
\end{equation}
where $\alpha = \abs{A}/N$ 
and 
$o(1)$ stands for some quantity that goes to $0$ as $N \to \infty$ 
    (not depending on $A$ and $d$).
Informally, the result says that one can always find some nonzero $d$ such that number of 3-term arithmetic progressions with common difference $d$ is approximately at least what one expects for a random subset $A \subseteq [N]$ with density $\alpha$. 
In contrast, there exist sets $A \subseteq [N]$ with density $\alpha$ such that the total number of 3-term arithmetic progressions in $A$ is at most $\alpha^{c \log(1/\alpha)} N^2$, much smaller than random (one can construct such sets by ``blowing up'' large subsets without 3-term arithmetic progressions).
Green developed an arithmetic analog of Szemer\'edi's regularity lemma to prove this result.
The same proof extends to other 3-point patterns, namely, for fixed positive integers $k_1 < k_2$, the conclusion remains true if \cref{eq:3-ap-popular} were replaced by
\begin{equation} \label{eq:3-pt-popular}
\abs{\set{t : t, t + k_1d, t+ k_2d \in A}} \ge (\alpha^3 - o(1)) N.
\end{equation}

Green and Tao \cite{GrTa10} extended the above result to 4-term arithmetic progressions, showing that for every $A \subseteq [N] := \{1, \dots, N\}$, there exists some $d \ne 0$ such that
\begin{equation} \label{eq:4-ap-popular}
\abs{\set{x : x, x+d, x+2d, x + 3d \in A}} \ge (\alpha^4 - o(1)) N,
\end{equation}
where $\alpha = \abs{A}/N$, as $o(1) \to 0$ as $N \to \infty$ as before. Their proof uses quadratic Fourier analysis. The same proof shows that for fixed positive integers $k_1 < k_2$, the conclusion remains true if \cref{eq:4-ap-popular} were replaced by
\begin{equation} \label{eq:4-pt-popular}
\abs{\set{t : t, t + k_1d, t + k_2d, t + (k_1 + k_2)d \in A}} \ge (\alpha^4 - o(1)) N.
\end{equation}
The above results were conjectured by Bergelson, Host, and Kra~\cite{BHKR}, who had proved weaker results using ergodic theory. Surprisingly, the extension to $k$-term arithmetic progressions is false for $k \ge 5$, as shown by a construction of Ruzsa~\cite{BHKR}.

Can the above popular common difference result hold for any other pattern?
In this article, we show that the answer is no. 

Let $P \subseteq \ZZ^r$ be finite set of points (a ``pattern'').
We call $r$ the \emph{ambient dimension} of $P$.
The dimension of the smallest affine subspace of $\RR^r$ containing $P$ is called the \emph{affine dimension} of $P$. 
For example, the pattern $P = \{(0,1), (1,1), (2,1)\}$ has affine dimension $1$ and ambient dimension $2$.

We define $\pdd_P(\alpha)$, the \emph{popular difference density for $P$ at density $\alpha$}, to be the largest possible real number such that for every $\epsilon > 0$, there exists $N_0 = N_0(P, \epsilon)$ such that for every $N \ge N_0$ and every $A \subseteq [N]^r$ with $\abs{A} \ge \alpha N^r$, there is some nonzero integer $d$ such that one has
\[
    \abs{\set{ x \in \ZZ^r : x + dy \in A \text{ for all } y \in P}} \ge (\pdd_P(\alpha) - \epsilon) N^r.
\]
We always have $\pdd_P(\alpha) \le \alpha^{\abs{P}}$ for every $0 < \alpha < 1$ and every $P$,
by considering a random subset of $[N]^r$ of density $\alpha$ as $N \to \infty$.
An easy argument\footnote{If we were working inside a group, e.g., $A \subseteq \ZZ/N\ZZ$, the claim that $\pdd_P(\alpha) = \alpha^{\abs{P}}$ for $\abs{P} = 2$ would follow trivially from averaging. However, since we are working with $A \subseteq [N]^{r}$, we need a small modification to restrict our attention to small differences.
For simplicity consider $P = \{0,1\} \subseteq \ZZ$; general two-point $P$ follows by an additional averaging argument.
Let $m \to \infty$ and $m = o(N)$. We have, by the Cauchy--Schwarz inequality,
\begin{align*}
\sum_{x,d_1,d_2 \in \ZZ} 1_A(x+d_1)1_A(x+d_2) 1_{[m]}(d_1) 1_{[m]}(d_2)
&= \sum_x \paren{\sum_d 1_A(x+d) 1_{[m]}(d)}^2 \\
&\ge \frac{1}{N+m} \paren{\sum_{x,d} 1_A(x+d) 1_{[m]}(d)}^2 = \frac{m^2}{(1+o(1))N } \abs{A}^2.
\end{align*}
So, by averaging, there exist a pair of distinct $d_1, d_2 \in [m]$ such that $|\{x : x, x+d_1-d_2 \in A\}| \ge \alpha^2 N - o(N)$.} shows that $\pdd_P(\alpha) = \alpha^{\abs{P}}$ if $\abs{P} \le 2$.

\begin{table}
    \centering
    \begin{tabular}{lll}
    $P \subseteq \ZZ^r$ &  Popular difference density  &  Reference \\
    \midrule 
    3 points in $\ZZ$ & $\pdd_P(\alpha) = \alpha^3$ & \cite{Green05} \\
    $k_0 < k_1 < k_2 < k_3$ in $\ZZ$ with $k_0 + k_3 = k_1 + k_2$ & $\pdd_P(\alpha) = \alpha^4$ & \cite{GrTa10} \\
    Other 4 point patterns in $\ZZ$ & $\pdd_P(\alpha) < (1-c) \alpha^4$ & \cref{thm:4pt-1dim} \\
    Affine dim of $P$ $<$ $r$ & $\pdd_P(\alpha) < \alpha^{c\log(1/\alpha)}$ & \cref{thm:corner-new} \\
    3 non-collinear points in $\ZZ^2$ & $\omega(\alpha^4) \le  \pdd_P(\alpha) \le \alpha^{4-o(1)}$ & \cite{Man,FSSSZ,Berger} \\
    4 points in strict convex position in $\ZZ^2$ & $\pdd_P(\alpha) < \alpha^{5-o(1)}$ & \cref{thm:4pt-2dim} \\
    4 points in nonconvex position in $\ZZ^2$ & $\pdd_P(\alpha) < \alpha^{c \log(1/\alpha)}$ & \cref{thm:4pt-2dim} \\
    At least 5 points & $\pdd_P(\alpha) < \alpha^{c \log (1/\alpha)}$ & \cite{BHKR,FSSSZ} \\
    Affine dimension at least 3 & $\pdd_P(\alpha) < \alpha^{c \log(1/\alpha)}$ & \cite{FSSSZ} \\
    \end{tabular}
    \caption{A summary of current bounds on the popular difference density $\pdd_P(\alpha)$.
    Here $c > 0$ depends on $P$.}
    \label{tab:summary}
\end{table}

Let us summarize old and new results. Also see \cref{tab:summary}.

The results of Green~\cite{Green05} and Green--Tao~\cite{GrTa10} discussed earlier can be rephrased as follows.

\begin{theorem}[\cite{Green05}] \label{thm:3ap}
    If $P \subseteq \ZZ$ with $\abs{P} = 3$, then $\pdd_P(\alpha) = \alpha^3$ for all $0 < \alpha < 1$.
\end{theorem}

\begin{theorem}[\cite{GrTa10}] \label{thm:4ap}
    If $P = \{k_0, k_1, k_2, k_3\}$ with integers $k_0 < k_1 < k_2 < k_3$ and $k_0 + k_3 = k_1 + k_2$,
    then $\pdd_P(\alpha) = \alpha^4$ for all $0 < \alpha < 1$.
\end{theorem}

Ruzsa's counterexample \cite{BHKR} showed that the above results do not extend to 5-term (or longer) arithmetic progressions.
His construction was extended to all patterns in $\ZZ$ with at least 5 points in \cite{FSSSZ}.

\begin{theorem}[\cite{BHKR,FSSSZ}] \label{thm:5pt}
    Let $P \subseteq \ZZ$ with $\abs{P} \ge 5$.
    Then there is some $c = c_P > 0$ so that $\pdd_P(\alpha) < \alpha^{c \log(1/\alpha)}$ for all $0 < \alpha < 1/2$.
\end{theorem}

Now let us move on to patterns in higher dimensions.
The first example of a truly higher-dimensional pattern is that of a ``corner'': $P = \{(0,0), (1,0), (0,1)\} \subseteq \ZZ^2$, which is essentially equivalent to the case of $P$ being three non-collinear points in $\ZZ^2$. 
In the finite field model (i.e., working inside $\FF_p^n$ for a fixed $p$ rather than in $[N]$ or $\ZZ/N\ZZ$),
Mandache \cite{Man} essentially reduced the popular common difference problem for corners to a certain extremal problem for 3-uniform hypergraphs. Berger \cite{Berger} extended Mandache's results to $[N]$ as well as arbitrary abelian groups of odd order.
Combined with \cite{FSSSZ}, which gave nearly tight upper and lower bounds on the associated extremal hypergraph problem (involving a 3-uniform hypergraph called the ``triforce''), we know the following. Here by $\omega(\alpha^4) \le  \pdd_P(\alpha)$ we mean that $\pdd_P(\alpha) / \alpha^4 \to \infty$ as $\alpha \to 0$.

\begin{theorem} \label{thm:corners}
    Let $P$ be three non-collinear points in $\ZZ^2$.
    Then $\omega(\alpha^4) \le  \pdd_P(\alpha) \le \alpha^{4-o(1)}$, where the asymptotics refer to $\alpha \to 0$.
\end{theorem}

The situation is dramatically different for corners in $\ZZ^r$ with $r \ge 3$. The following result is shown in \cite{FSSSZ}. We give a new proof of this theorem that is easier than the one in \cite{FSSSZ}.

\begin{theorem}[\cite{FSSSZ}] \label{thm:3dim}
    Let $P \subseteq \ZZ^r$ with affine dimension at least $3$.
    Then there is some $c = c_P > 0$ so that $\pdd_P(\alpha) < \alpha^{c \log(1/\alpha)}$ for all $0 < \alpha < 1/2$.
\end{theorem}

Now let us discuss new results. First, let us consider 1-dimensional patterns. 
Let $P \subseteq \ZZ$.
It is not hard to see that  $\pdd_P(\alpha) = \alpha^{\abs{P}}$ if $\abs{P} \le 2$.
From \cref{thm:3ap} we know that $\pdd_P(\alpha) = \alpha^{\abs{P}}$ if $\abs{P} = 3$. \cref{thm:5pt} shows that $\pdd_P(\alpha) < \alpha^{c\log(1/\alpha)}$ whenever $\abs{P} \ge 5$. It remains to study 4-point patterns.
\cref{thm:4ap} shows that $\pdd_P(\alpha) = \alpha^{\abs{P}}$ for $P = \{k_1, k_2, k_3, k_4\}$ with $k_1 < k_2 < k_3 < k_4$ and $k_1 + k_4 = k_2 + k_3$. It remains to study 4-point patterns in $\ZZ$ not of this form, and our next result shows that $\pdd_P(\alpha) < (1-c)\alpha^4$. 
See \cref{sec:4pt-1d} for proof, which uses computer assistance.

\begin{theorem}[4-point 1-dimensional patterns] \label{thm:4pt-1dim}
There is some absolute constant $c > 0$ such that for all $P \subseteq \ZZ$ with $|P|  =4$ and not of the form $P = \{k_0, k_1, k_2, k_3\}$ with integers $k_0 < k_1 < k_2 < k_3$ and $k_0 + k_3 = k_1 + k_2$, one has $\pdd_P(\alpha) < (1-c)\alpha^4$ for all $0 < \alpha < 1/2$.
\end{theorem}

In some cases, we can prove even better bounds, as stated next. 
For example, there exist $P \subseteq \ZZ$ with $\abs{P} = 4$ and $\pdd_P(\alpha) < \alpha^{100}$ for all sufficiently small $\alpha > 0$. 
See \cref{sec:4pt-1d-special} for proof.

\begin{theorem}[Certain 4-point 1-dimensional patterns] \label{thm:4pt-special-1dim}
For every $C>0$ there exists some $P \subseteq \ZZ$ with $\abs{P} = 4$ such that $\pdd_P(\alpha) < \alpha^C$ for all sufficiently small $\alpha > 0$.
\end{theorem}

Now let us move to higher-dimensional patterns. \cref{thm:corners} shows that $\pdd_P(\alpha) = \alpha^{4-o(1)}$ for every $P \subseteq \ZZ^2$ with $\abs{P} = 3$ and affine dimension $2$. By \cref{thm:5pt}, $\pdd_P(\alpha) < \alpha^{c\log(1/\alpha)}$ whenever $\abs{P} \ge 5$. For 4-point patterns in $\ZZ^2$, we obtain the following upper bounds, whose proof can be found in \cref{sec:2dim,sec:nonconvex}.

\begin{theorem}[4-point 2-dimensional patterns] \label{thm:4pt-2dim}
Let $P \subseteq \ZZ^2$ with $\abs{P} = 4$.
\begin{enumerate}
    \item If $P$ is 4 points in strict convex position, then $\pdd_P(\alpha) < \alpha^{5-o(1)}$, where the $o(1)$ is some quantity that goes to zero as $\alpha \to 0$.
    \item Otherwise, there is some $c = c_P > 0$ such that $\pdd_P(\alpha) < \alpha^{c \log(1/\alpha)}$ for all $0 < \alpha < 1/2$.
\end{enumerate}
\end{theorem}

The next statement tells us what happens when $P \subseteq \ZZ^r$ is not full-dimensional. See \cref{sec:ambient} for proof.

\begin{theorem}\label{thm:corner-new}
Let $P \subseteq \ZZ^r$ with $\abs{P} \ge 3$ and suppose that the affine dimension of $P$ is strictly less than its ambient dimension $r$.
Then there exists some $c = c_P > 0$ such that
$\pdd_P(\alpha) < \alpha^{c \log(1/\alpha)}$ for all $0 < \alpha < 1/2$.
\end{theorem}

\cref{thm:corner-new} gives a new proof of \cref{thm:3dim}.
Indeed, if $P \subseteq \ZZ^r$ has affine dimension at least 3, then let $P' \subseteq P$ be an arbitrary 3-point subset. 
Then the affine dimension of $P'$ is at most 2, 
and hence $\pdd_{P'}(\alpha) < \alpha^{c \log(1/\alpha)}$ by \cref{thm:corner-new}.
Note from definition that $\pdd_P(\alpha) \le \pdd_{P'}(\alpha)$, 
and thus $\pdd_{P}(\alpha) < \alpha^{c' \log(1/\alpha)}$.

Putting all of the above results together, we find that no other patterns $P$ with $\abs{P} \ge 3$ satisfy \cref{thm:3ap,thm:4ap}.

\begin{corollary} \label{cor:all}
Let $P \subseteq \ZZ^r$ with $\abs{P} \ge 3$. Unless $r = 1$ and $P$ is one of the sets in \cref{thm:3ap,thm:4ap}, we have $\pdd_P(\alpha) < \alpha^{\abs{P}}$ for all sufficiently small $\alpha > 0$.
\end{corollary}

We do not give any new lower bounds on $\pdd_P(\alpha)$ in this paper. 
Except in the cases addressed by \cref{thm:3ap,thm:4ap,thm:corners}, 
    the best lower bounds that we are aware of essentially come from quantitative bounds on the multidimensional Szemer\'edi theorem.
    Indeed, the multidimensional Szemer\'edi theorem \cite{FK78} implies that for every finite $P \subseteq \ZZ^r$ and $\alpha > 0$ there is some $c_P(\alpha) > 0$ so that every subset of $[N]^r$ with density $\alpha$ contains at least $c_P(\alpha)N^{r+1}$ copies of $P$ (allowing translations and dilations), which then by an averaging argument implies that $\pdd_P(\alpha) \ge c_P(\alpha)$.
For all $P$ with at least 4 points and affine dimension at least 2, 
    the best bounds on the multidimensional Szemer\'edi theorem comes from the hypergraph removal lemma \cite{Gow07,RS06}. For 3 non-collinear points, such as the corners pattern, the best bound is due to Shkredov \cite{Shk06}. 

It remains interesting to improve the bounds further, especially for \cref{thm:4pt-1dim,thm:4pt-2dim}.

\medskip

\noindent \textbf{Acknowledgments.} The third author would like to thank Ben Green for hosting him during a visit to Oxford and for discussions that led to this project.

\section{Patterns whose affine dimension is less than its ambient dimension}\label{sec:ambient}
In this section we prove \cref{thm:corner-new}. The following proposition is a well-known application of Behrend's construction of large subsets without 3-AP arithmetic progressions.

\begin{proposition}\label{prop:behrend-3pt}
Let $P\subseteq\ZZ^{r}$ and $|P|\ge 3$ and fix $0< \alpha<1/2$. Then there exists some $c = c_P > 0$ such that for all sufficiently large $N$, there exists $S\subseteq(\ZZ/N\ZZ)^r$ such that $S$ contains at most $\alpha^{c_P\log(1/\alpha)}N^{r+1}$ translated dilates of $P$ and $|S|\ge \alpha N^r$.
\end{proposition}
\begin{proof}[Proof sketch]
By an appropriate generalization of Behrend's construction~\cite{B}, there is a subset $\Lambda\subseteq [L]^r$ of size $|\Lambda|\ge L^r\exp(-c_P\sqrt{\log L})$ avoiding translated dilates of $P$. For example, by taking $\Lambda$ to be the inverse image of an appropriate set $\Lambda'$ under linear projection to $1$ dimension, we can reduce to the case $r = 1$. This case is directly handled by standard modifications of Behrend's construction.

Then essentially blowing up each point into a box of widths $\lfloor N/L\rfloor$ gives the desired result. For correctness' sake, one must only use the middle $1/C_P$ fraction of this box (for appropriately chosen $C_P > 0$) to force all translated dilates of $P$ to stay within a box (using the property of $\Lambda$ that it avoids translated dilates of $P$).
\end{proof}
Finally, it will be useful to have an explicit relationship between patterns that are related via an affine-linear transformation.
\begin{proposition}\label{prop:basis-change}
Let $P,Q\subseteq\ZZ^r$ be such that there is an invertible affine-linear transformation $\phi: \QQ^r\to\QQ^r$ satisfying $\phi(P) = Q$. Then there is a constant $c=c_{P,Q}\in(0,1)$ such that
\[\pdd_Q(c\alpha)\le\pdd_P(\alpha).\]
\end{proposition}
\begin{proof}
For every $\epsilon > 0$ 
and sufficiently large $N$, 
we can find a set $A\subseteq [N]^r$ which satisfies
\[
\max_{d\ne 0}
\abs{\set{ x \in \ZZ^r : x + dy \in A \text{ for all } y \in P}} \le (\pdd_P(\alpha) + \epsilon) N^r
\]
and
\[|A|\ge \alpha N^r.\]
We consider $\phi(A)$. As $\phi$ is an invertible linear map $\QQ^r\to\QQ^r$ we have that
\[\phi([N]^{r})\subseteq \cup_{i=1}^{c_\phi}([-s_\phi N,s_\phi N]^{r}+Y_i)\]
for some points $Y_i\in \QQ^{r}$ and some positive integers $c_\phi,s_\phi$ depending only on $\phi$. That is, $\phi$ maps $[N]^r$ maps into a bounded number of rational translates of $[-s_\phi N,s_\phi N]^{r}$. By pigeonholing, there exists $i$ such that
\[\big|\phi(A)\cap ([-s_\phi N,s_\phi N]^{r}+Y_i)\big|\ge |A|/((3s_\phi)^rc_\phi).\]
Let $A' = -Y_i + \phi(A)\cap [-s_\phi N,s_\phi N]^{r}$. Now by construction
\[
\max_{d\ne 0}
\abs{\set{ x \in \ZZ^r : x + dy \in A' \text{ for all } y \in Q}} \le (\pdd_P(\alpha) + \epsilon) N^r\le (\pdd_P(\alpha) + \epsilon) (2s_\phi N+1)^r
\]
and
\[|A'|\ge \alpha/((3s_\phi)^rc_\phi) \cdot N^r\ge \alpha/((3s_\phi)^{2r}c_\phi)\cdot (2s_\phi N+1)^r.\]
This implies the desired result. In particular, we can take $c = 1/((3s_\phi)^{2r}c_\phi)$.
\end{proof}

Using these propositions we can now easily prove \cref{thm:corner-new}.
\begin{proof}[Proof of \cref{thm:corner-new}]
We can assume $N$ is prime, up to losing at most an absolute constant factor by Bertrand's postulate. It also suffices to perform the construction in $(\ZZ/N\ZZ)^r$.

Let $P\subseteq\ZZ^r$ have affine dimension of $r' < r$. Then \cref{prop:basis-change} shows that, up to losing at most a constant factor, we can apply an invertible affine transformation to obtain a different pattern. (We will often perform this step implicitly in the future.) In particular, we can reduce to the case where $P$ spans precisely the first $r'$ coordinate directions. Since $|P|\ge 3$, we can find a subset $S$ of $(\ZZ/N\ZZ)^{r'}$ with density $\alpha$ and $\alpha^{c_P\log(1/\alpha)}N^{r'+1}$ translated dilates of $P$ by \cref{prop:behrend-3pt}. Taking the set
\[A = \{(i_1\cdot s_1,\ldots,i_1\cdot s_{r'}, i_1,i_2,\ldots,i_{r-r'}) : i_1 \neq 0, i_j\in \ZZ/N\ZZ, s=(s_1,\ldots,s_{r'})\in S\}\subseteq(\ZZ/N\ZZ)^r,\]
the result follows as the number translates of $P$ with a common difference $d$ is precisely the number of translated dilates of $P$ in $S$ times $N^{r-r'-1}$. (This is because every difference $d$ occurs an equal amount of times, since the construction includes a dilate of $S$ by every possible factor $i_1\in(\ZZ/N\ZZ)^\times$.) The result follows. 
\end{proof}

\section{Four-point patterns in two dimensions} \label{sec:2dim}
We now consider two-dimensional four-point patterns with the four points in strict convex position. This proof extends an earlier construction of Mandache \cite{Man}, and takes place in a more general context of a finite abelian group $G\times G$ rather than $[N]^2$. Assuming that the order of the group $G$ is relatively prime to a certain integer, we can replace our patterns with $(g,h),(g+d,h),(g,h+d),(g+k_1d,h+k_2d)$ where $k_1,k_2\in\QQ_{> 0}$ via rescaling. (Specifically, if $|G|$ is relatively prime to the product of the denominators of $k_1$ and $k_2$ then multiplication of an element of $G$ by $k_1,k_2$ is well-defined.) Note that $k_1+k_2\neq 1$. Taking $G = \ZZ/N\ZZ$ then transferring the resulting set $S$ to $[N]$, we immediately deduce the first part of \cref{thm:4pt-2dim}.
\begin{theorem}\label{thm:sec3-main}
Fix a pair of rationals $(k_1,k_2)\in\QQ_{> 0}$. There exists some constant $C > 0$ so that for all $0 < \alpha < 1/2$ and all abelian groups of order $N > N_0(\alpha,k_1,k_2)$ relatively prime to some $M(k_1,k_2)$, the following holds. There exists some $S\subseteq G\times G$ with $|S|\ge\alpha|G|$ so that for every $d\ne 0$ we have
\[
\EE_{x,y}\one_S(x,y)\one_S(x+d,y)\one_S(x,y+d)\one_S(x+k_1d,y+k_2d) < \alpha^5e^{C\sqrt{\log(1/\alpha)}}.
\]
\end{theorem}

We have a finite abelian group $G$ of order relatively prime to some constant $M(k_1,k_2)$. Let $f: [0,1]^3\to [0,1]$ be piecewise continuous, to be chosen later. Sample $\mathbf{X} = (X_g)_{g\in G}$, $\mathbf{Y}=(Y_g)_{g\in G}$, and $\mathbf{Z} = (Z_g)_{g\in G}$ uniformly from $[0,1]^G$. Let $F: G\times G\to [0,1]$ be a random function defined via
\[F(g, h) = f(X_g,Y_h,Z_{g+h}).\]
For nonzero $d\in G$ define
\begin{align*}
\alpha(F) &:= \EE_{g,h}F(g,h) \qquad \text{and}\\
\beta(F,d) &:= \EE_{g,h}F(g,h)F(g+d,h)F(g,h+d)F(g+k_1h,d+k_2h),
\end{align*}
which are random variables. Then define $\alpha$ to be 
\[\alpha = \EE_{\mathbf{X},\mathbf{Y},\mathbf{Z}}\alpha(F) = \EE_{g,h}\EE_{\mathbf{X},\mathbf{Y},\mathbf{Z}}F(g,h) = \EE_{x,y,z}f(x,y,z).\]
The last equality is true since the inner expectation over $\mathbf{X},\mathbf{Y},\mathbf{Z}$ is independent of $g,h$ and equals the right hand side. Define $\beta(d)$ to be
\begin{align}
\beta(d) 
&= \EE_{\mathbf{X},\mathbf{Y},\mathbf{Z}}\beta(F,d) \notag 
\\
&= \EE_{g,h}\EE_{\mathbf{X},\mathbf{Y},\mathbf{Z}}[f(X_g,Y_h,Z_{g+h})f(X_{g+d},Y_h,Z_{g+h+d})\notag\\
&\qquad\qquad\qquad f(X_g,Y_{h+d},Z_{g+h+d})f(X_{g+k_1d},Y_{h+k_2d},Z_{g+h+(k_1+k_2)d})]\notag  \\
&=\EE f(x_0,y_0,z_0)f(x_1,y_0,z_1)f(x_0,y_1,z_1)f(x_{k_1},y_{k_2},z_{k_1+k_2})
\end{align}
where in the final expression, the $x_i$, $y_i$, $z_i$'s are all iid uniform random variables in $[0,1]$. Indeed, the final equality holds even if $g$ and $h$ were held fixed at arbitrary values in the second-to-last line. This step uses the hypothesis that $|G|$ is relatively prime to the nonzero elements of $\{k_1-1,k_2-1,k_1+k_2-1\}$.

Note that $\beta = \beta(d)$ thus is independent of the value $d\neq 0$. Now, for a set $S$ we define the analogous notions
\begin{align*}
\alpha(S) &= \EE_{g,h}\mbm{1}_S(g,h)\text{ and}\\
\beta(S,d) &=  \EE_{g,h}\mbm{1}_S(g,h)\mbm{1}_S(g+d,h)\mbm{1}_S(g,h+d)\mbm{1}_S(g+k_1d,h+k_2d).
\end{align*}
Now sample a random subset $S$ of $G\times G$ by sampling each pair $(g,h)$ with probability $F(g,h)$. We show that as $N\to\infty$, the size of $S$ and the number of squares in $S$ of difference $d$ concentrate around their mean values $\alpha = \EE_{\mathbf{X},\mathbf{Y},\mathbf{Z}}\alpha(F)$ and $\beta = \EE_{\mathbf{X},\mathbf{Y},\mathbf{Z}}\beta(F,d)$. This reduces the problem to constructing $f$ with $\EE f = \alpha$ such that
\[\EE f(x_0,y_0,z_0)f(x_1,y_0,z_1)f(x_0,y_1,z_1)f(x_{k_1},y_{k_2},z_{k_1+k_2}) = \beta\]
is minimized.

In order to obtain concentration we will require the bounded difference inequality (see \cite[Theorem~6.2]{BLM13}). 
\begin{theorem}
Suppose that $f:\mathcal{X}^n\to \RR$ satisfies that 
\[\sup_{x_1,\ldots,x_n,x_i'\in \mc{X}}|f(x_1,\ldots,x_i,\ldots,x_n)-f(x_1,\ldots,x_i',\ldots,x_n)|\le c_i.\] Then if $X_1,\ldots, X_n$ are independent then $Z = f(X_1,\ldots,X_n)$ satisfies
\[\PP[|Z-\EE[Z]|\ge \epsilon]\le \exp\bigg(-\frac{2\epsilon^2}{\sum_{i=1}^{k}c_i^2}\bigg)\]
\end{theorem}

\begin{lemma}\label{lem:convex-concentration}
Fix a function $f:[0,1]^3\to [0,1]$. Sample a random subset $S$ of $G\times G$ by sampling $X_g,Y_g,Z_g$ uniform from $[0,1]$ (independently for all $g\in G$) and then include each pair $(g,h)$ in $S$ with probability $F(g,h) = f(X_g,Y_h,Z_{g+h})$. Then with probability $1-o(1)$ as $|G|\to\infty$ we have
\[|\alpha(S)-\alpha| \le |G|^{-1/3}\]
and
\[\sup_{d\neq 0}|\beta(S,d)-\beta| \le |G|^{-1/3}.\]
\end{lemma}
\begin{proof}
Let $N = |G|$ and we let $\mathbf{W} = (W_{g,h})_{g,h\in G}$ be a set of independent uniform $[0,1]$ random variables. We see that the random set $S$ is a function of the random variables $\mathbf{X}$, $\mathbf{Y}$, $\mathbf{Z}$, and $\mathbf{W}$ as follows: $(g,h)\in S$ if and only if $f(X_g,Y_h,Z_{g+h})\ge W_{g,h}$. Thus $\alpha(S)$ and $\beta(S,d)$ can be expressed as $(N^2+3N)$-variate function of the random variables $\mathbf{X}$, $\mathbf{Y}$, $\mathbf{Z}$, and $\mathbf{W}$. We will apply the bounded difference inequality to prove the desired concentration.

If we consider $S$ as a function of $(\mathbf{X},\mathbf{Y},\mathbf{Z},\mathbf{W})$, note that changing any single value of $X_g$, $Y_g$, or $Z_g$ changes at most $N$ elements of $S$, and changing any $W_{g,h}$ affects at most $1$ element of $S$. Therefore any change will alter
\[\alpha(S) = \EE_{g,h}\mbm{1}_S(g,h)\]
by at most $1/N$ for changing any of $X_g,Y_g,Z_g$ or $1/N^2$ for $W_{g,h}$. Similarly, changing any $X_g,Y_g,Z_g$ will change $\beta(S,d)$ by at most $O(1/N)$ and changing any $W_{g,h}$ will change it by at most $O(1/N^2)$. The bounded difference inequality shows that $\alpha(S)$ and $\beta(S,d)$ lie within $\delta N^{-1/2}$ of their means with probability $1-\exp(-\Omega(\delta^2))$. Choosing $\delta = N^{1/6}$ and taking a union bound over nonzero $d\in G$ gives the result.
\end{proof}
We are now in position to prove \cref{thm:sec3-main}.
\begin{proof}[Proof of \cref{thm:sec3-main}]
By \cref{lem:convex-concentration} it suffices to define an appropriate function $f$ with
\[\EE_{x,y,z} f(x,y,z)\in[\alpha,3\alpha/2]\]
which satisfies
\begin{equation}\label{eq:beta-is-small}
\beta = \EE f(x_0,y_0,z_0)f(x_1,y_0,z_1)f(x_0,y_1,z_1)f(x_{k_1},y_{k_2},z_{k_1+k_2}) < \alpha^5e^{C\sqrt{\log(1/\alpha)}}.
\end{equation}

Now we choose an appropriate function $f$. Let $H$ be a triparite graph defined with vertex sets $X=Y=Z=\ZZ/L\ZZ$. Let $\Lambda$ be a subset of $\ZZ/L\ZZ$ avoiding 3-term arithmetic progressions with $|\Lambda|=\lfloor Le^{-C\sqrt{\log L}}\rfloor$ for an absolute constant $C > 0$, whose existence is due to Behrend~\cite{B}. Let $H$ have edges $(x,x+a)\in X \times Y$, $(y,y+a)\in Y\times Z$ and $(x,x+2a)\in X\times Z$ for $x,y\in \ZZ/L\ZZ$ and $a\in\Lambda$. Note that since $\Lambda$ is $3$-AP free the only triangles in $H$ are of the form $(x,x+a,x+2a)\in X\times Y\times Z$. Therefore no two triangles share an edge, there are $L|\Lambda|$ triangles, and any vertex is in $|\Lambda|$ triangles. We let $f(x,y,z) = 1$ if $(\lfloor Lx\rfloor,\lfloor Ly\rfloor,\lfloor Lz\rfloor)/L$ is a triangle in $H$, and $0$ otherwise.

Now we split into cases. Recall $k_1,k_2\in\QQ_{>0}$. Furthermore $k_1 + k_2\ne 1$, as otherwise this would not be strictly convex. There is also a symmetry in $k_1$ and $k_2$, so it suffices to prove \cref{eq:beta-is-small} in the cases (1) $k_1,k_2\ne 1$, (2) $k_1 = 1$ and $k_2\neq 1$, and (3) $(k_1,k_2)= (1,1)$.

\begin{enumerate}
    \item We have
    \begin{align*}
    \beta &= \EE_{\substack{x_0,x_1,x_2\\y_0,y_1,y_2\\z_0,z_1,z_2}}f(x_0,y_0,z_0)f(x_1,y_0,z_1)f(x_0,y_1,z_1)f(x_2,y_2,z_2) = \frac{L^2|\Lambda|^2}{L^9} = \frac{|\Lambda|^2}{L^7}.
    \end{align*}
    To justify this, we count the number of tuples $(\mathbf{x},\mathbf{y},\mathbf{z})$ which make the inner term equal $1$ (else it is $0$), which occurs precisely when the four triples that appear in the express above are all triangles, in which case $x_0y_0z_1$ must also be a triangle. But no two triangles in $H$ share an edge, which forces $z_0 = z_1$ and $y_0 = y_1$ and $x_0 = x_1$. The number of choices of variables that make $(x_0,y_0,z_0)$ and $(x_2,y_2,z_2)$ both triangles is $L^2 \abs{\Lambda}^2$.
    
    \item We have
    \begin{align*}
    \beta &= \EE_{\substack{x_0,x_1\\y_0,y_1,y_2\\z_0,z_1,z_2}}f(x_0,y_0,z_0)f(x_1,y_0,z_1)f(x_0,y_1,z_1)f(x_1,y_2,z_2) = \frac{L|\Lambda|^2}{L^8} = \frac{|\Lambda|^2}{L^7}
    \end{align*}
    for the same reason, except that we obtain two vertex-attached triangles $(x_0,y_0,z_0)$ and $(x_0,y_2,z_2)$. Since every vertex is in $|\Lambda|$ triangles, there are $L|\Lambda|^2$ such configurations.
    
    \item We have
    \begin{align*}
    \beta&= \EE_{\substack{x_0,x_1\\y_0,y_1\\z_0,z_1,z_2}}f(x_0,y_0,z_0)f(x_1,y_0,z_1)f(x_0,y_1,z_1)f(x_1,y_1,z_2) = \frac{L|\Lambda|}{L^7} = \frac{|\Lambda|}{L^6}.
    \end{align*}
    Again we find that the expression in the expectation is 1 if and only if $x_0=x_1$, $y_0=y_1$, and $z_0=z_1$, in which case since $(x_1,y_1,z_1)$ and $(x_1,y_1,z_2)$ must be the same triangle since they share an edge and so $z_1 = z_2$. Thus we obtain $L|\Lambda|$ configurations.
    
\end{enumerate}
We are now in a position to establish \cref{eq:beta-is-small} for all cases simultaneously. We choose $L$ such that
\[\frac{\lfloor Le^{-C\sqrt{\log{L}}}\rfloor}{L^2} = \frac{|\Lambda|L}{L^3}\in[\alpha,3\alpha/2].\]
This is easily seen to be feasible, and furthermore such a choice implies that $\log(\alpha L)/\sqrt{\log(1/\alpha)}\in (-c^{-1},-c)$ for some absolute constant $c\in(0,1)$. In particular,
\[L\ge\alpha^{-1}e^{-c^{-1}\sqrt{\log(1/\alpha)}}.\]
Now, regardless of which case we are in, we obtain
\[\beta\le\frac{|\Lambda|}{L^6}\le \frac{e^{-C\sqrt{\log{L}}}}{L^5}\le\alpha^5e^{C'\sqrt{\log (1/\alpha)}}.\qedhere\]
\end{proof}

\section{Nonconvex patterns in two dimensions} \label{sec:nonconvex}
In this section we prove that all nonconvex four point patterns $P$ satisfy $\pdd_P(\alpha)<\alpha^{c\log(1/\alpha)}$ for all $\alpha\in(0,1/2)$, for some appropriate constant $c = c_P > 0$. The proof is a variant of the construction showing $\pdd_{\{0,1,2,3,4\}}(\alpha)<\alpha^{c\log(1/\alpha)}$ in \cite[Appendix]{BHKR} as well as the construction showing three-dimensional corners satisfy $\pdd_P(\alpha)<\alpha^{c\log(1/\alpha)}$ that establishes \cite[Theorem~1.6]{FSSSZ}. However, carrying out the ``natural analog'' of these constructions would require a subset of $[N]$ of size $N^{1-o(1)}$ avoiding an equation such as $2x+2y=3z+w$; it is unknown whether such sets exist. We overcome this obstacle by a novel  extension of these constructions using complex numbers.

Most of the second part of \cref{thm:4pt-2dim} is implied by the following theorem (only the case where three points are collinear is left out, which is handled at the end of this section).
\begin{theorem}\label{thm:2d-general}
Let $P\subseteq\ZZ^2$ be a set of four points in strictly nonconvex position. Let $0 < \alpha < 1/2$. For all sufficiently large $N$, there exists $A \subseteq [N]^2$ with $|A| \ge \alpha N^2$ such that for all nonzero integers $d$, there are at most $\alpha^{c \log(1/\alpha)} N^2$ points $x \in \ZZ^2$ such that $x + d\cdot P := \{x + dt : t \in P\} \subseteq A$, where $c = c_P > 0$ is a constant.
\end{theorem}

By a change of basis via \cref{prop:basis-change}, we reduce \cref{thm:2d-general} to patterns of the form $P = \{(0,0),(m_1,0),(0,m_2),(-m_3,-m_4)\}$, with positive integers $m_1,m_2,m_3,m_4$. 
Let $m = m_2m_3+m_1m_4+m_1m_2$.

Let nonzero $A, B, C\in\CC$ such that $BC(B-C) = m_2m_3$, $CA(C-A) = m_1m_4$, and $AB(A-B) = m_1m_2$. It follows that $m_2m_3A + m_1m_4B + m_1m_2C = 0$. We justify the existence of such numbers.
\begin{lemma}
There exist nonzero $A, B, C\in\CC$ with $B/A\notin\RR$ such that $BC(B-C) = m_2m_3$, $CA(C-A) = m_1m_4$, and $AB(A-B) = m_1m_2$.
\end{lemma}
\begin{proof}
Let $R = m_1m_2\sqrt{-m_3m_4/m}$, which is nonzero and purely imaginary. Let $u,v,w$ be nonzero complex numbers satisfying
\begin{align*}
v-w&=\frac{m_2m_3}{R}u,\\
w-u&=\frac{m_1m_4}{R}v,\\
u-v&=\frac{m_1m_2}{R}w.
\end{align*}
For example, we can choose $u=1-(m_1m_4/R)$, $v = 1+(m_2m_3/R)$, and $w = 1+(m_1m_2m_3m_4/R^2)$, which by design satisfy the first two equations and satisfy the third by the definition of $R$.

We must check that these are nonzero. Since $R$ is purely imaginary, $u,v\neq 0$ is clear. Furthermore, if $w=0$ then $R^2=-m_1m_2m_3m_4$ or $m = m_1m_2$, which is a contradiction as $m_1,m_2,m_3,m_4$ are positive integers.

Now choose $t$ such that $t^3=R/(uvw)$. Let $(A,B,C)=t(u,v,w)$. Then $ABC = R$ and
\[
B-C=\frac{m_2m_3}{R}A, \quad
C-A=\frac{m_1m_4}{R}B, \quad \text{and}\quad 
A-B=\frac{m_1m_2}{R}C.
\]
Hence using $R = ABC$ gives
\[
BC(B-C) = m_2m_3,\quad CA(C-A) = m_1m_4,\quad \text{and}\quad  AB(A-B) = m_1m_2.
\]
Finally, we show that our choice yields $B/A\notin\RR$. Assume for the sake of contradiction that $B/A\in\RR$. Adding the above three linear equations gives
\[m_2m_3A+m_1m_4B+m_1m_2C = 0,\]
so if $B/A\in\RR$ then $C/A\in\RR$. Then
\[\frac{m_2m_3}{A^3} = \frac{BC(B-C)}{A^3}\in\RR,\]
thus $A^3\in\RR$. Let $A = a\exp(2\pi ij/3)$ for $a\in\RR$ and $j\in\{0,1,2\}$ chosen appropriately. Then since $B/A,C/A\in\RR$ we see that $b = B\exp(-2\pi ij/3)$ and $c = C\exp(-2\pi ij/3)$ are also real. Note that $a,b,c$ also satisfy $bc(b-c)=m_2m_3$, $ca(c-a)=m_1m_4$, and $ab(a-b)=m_1m_2$, as well as $m_2m_3a+m_1m_4b+m_1m_2c = 0$.

Since $m_2m_3a + m_1m_4b + m_1m_2c = 0$, the numbers $a,b,c$ do not all have the same sign. If we have $a,b>0>c$ then $bc(b-c) < 0$, and similar for the other three cyclic cases. If $c > 0 > a, b$ then $ca(c-a) < 0$, and similar for the other three cyclic cases. This gives the desired contradiction.
\end{proof}
Now fix a choice of such $A,B,C\in\CC$. Define
\begin{equation}\label{eq:general-f-def}
f(x, y) = \frac{(m_2Ax+m_1By)^2}{A}.
\end{equation}
The function $f$ satisfies the following identity.

\begin{lemma}\label{lem:alg-id}
Let $m_1,m_2,m_3,m_4, f$ be as above. For all $n_1,n_2,d$ we have 
\begin{align*}
m_2m_3f(n_1+m_1d,n_2) &+ m_1m_4f(n_1,n_2+m_2d) + m_1m_2f(n_1-m_3d,n_2-m_4d)\\
&= (m_2m_3+m_1m_4+m_1m_2)f(n_1,n_2).
\end{align*}
\end{lemma}
\begin{proof}
Note that relation
\[BC(B - C)(t_1 + At_2)^2 + CA(C - A)(t_1 + Bt_2)^2 + AB(A - B)(t_1 + Ct_2)^2 + (A - B)(B - C)(C - A)t_1^2 = 0\]
holds as an polynomial identity in the variables $A,B,C,t_1,t_2$, by expansion. Now recall that for our specific choices of $A,B,C\in\CC$ we have $BC(B-C) = m_2m_3$, $CA(C-A) = m_1m_4$, and $AB(A-B) = m_1m_2$. Summing these relations gives $(A-B)(B-C)(C-A) = -m_2m_3 - m_1m_4 - m_1m_2$. Substituting this in we obtain 
\[m_2m_3(t_1 + At_2)^2 + m_1m_4(t_1 + Bt_2)^2 + m_1m_2(t_1 + Ct_2)^2 = (m_2m_3 + m_1m_4 + m_1m_2)t_1^2.\] Now setting
$t_1 = m_2An_1+m_1Bn_2$ and $t_2 = m_1m_2d$ we obtain
\begin{multline*}
m_2m_3(m_2An_1+m_1Bn_2 + m_1m_2Ad)^2 
+ m_1m_4(m_2An_1+m_1Bn_2 + m_1m_2Bd)^2
\\
+ m_1m_2(m_2An_1+m_1Bn_2 + m_1m_2Cd)^2 
= (m_2m_3 + m_1m_4 + m_1m_2)(m_2An_1+m_1Bn_2)^2.
\end{multline*}
Recalling that $m_2m_3A + m_1m_4B + m_1m_2C = 0$ and dividing the expression by $A$ we obtain
\begin{multline*}
m_2m_3(m_2An_1+m_1Bn_2 + m_1m_2Ad)^2/A + m_1m_4(m_2An_1+m_1Bn_2 + m_1m_2Bd)^2/A\\
\qquad + m_1m_2(m_2An_1+m_1Bn_2 -m_2m_3Ad - m_1m_4Bd)^2/A 
\\
= (m_2m_3 + m_1m_4 + m_1m_2)(m_2An_1+m_1Bn_2)^2/A.
\end{multline*}
This is easily seen to be equivalent to
\begin{multline*}
m_2m_3f(n_1+m_1d,n_2) + m_1m_4f(n_1,n_2+m_2d) + m_1m_2f(n_1-m_3d,n_2-m_4d)\\
= (m_2m_3+m_1m_4+m_1m_2)f(n_1,n_2),
\end{multline*}
as desired.
\end{proof}

\begin{lemma}\label{lem:general-behrend-modified}
Let $m_1,m_2,m_3,m_4\in\ZZ_{>0}$. There is an absolute constant $c = c_{m_1,m_2,m_3,m_4} > 0$ such that the following holds. For every integer $L > 0$ there exists a subset $\Lambda$ of $\{0, 1, \ldots, L - 1\}$ having at least $L\exp(-c\sqrt{\log L})$ elements that does not contain any nontrivial solutions to $m_2m_3x+m_1m_4y+m_1m_2z=(m_2m_3+m_1m_4+m_1m_2)w$ (here a trivial solution is one with $x=y=z=w$).
\end{lemma}

\begin{proof}
This follows from a standard modification from  Behrend's construction~\cite{B} of a large 3-AP-free set (e.g., see \cite[Lemma 3.1]{Alon}).
\end{proof}

The next lemma is similar to \cite[Lemma~2.3]{BHKR}.
\begin{lemma}\label{lem:general-alpha-transfer}
Let $m_1, m_2, m_3, m_4, m, A, B, C$ be as above. Let $\Lambda$ be a subset of $\{0, 1, \ldots, L - 1\}$ not containing any nontrivial solutions to $m_2m_3x+m_1m_4y+m_1m_2z=mw$, and let $\psi$ be a fixed complex constant. For each $j = (j_1,j_2)\in\Lambda^2$, let
\[I_j := A\left[\frac{j_1}{mL}, \frac{j_1}{mL} + \frac{1}{m^2L}\right)+B\left[\frac{j_2}{mL}, \frac{j_2}{mL} + \frac{1}{m^2L}\right) \subseteq\CC/(A\ZZ + B\ZZ),
\]
and let
\[
\mc{B} = \bigcup_{j \in \Lambda^2} I_j.
\]
Let $f$ be defined by \cref{eq:general-f-def}. Let $n_1,n_2,d\in\ZZ$ and let $w = \psi f(n_1,n_2)$, $x = \psi f(n_1+m_1d,n_2)$, $y = \psi f(n_1,n_2+m_2d)$, and $z = \psi f(n_1-m_3d,n_2-m_4d)$. Suppose that $w, x, y, z \pmod{A\ZZ + B\ZZ}$ all lie in $\mc{B}$. Then 
\[
\norm{2m_1m_2(m_2An_1 + m_1Bn_2)d\psi + m_1^2m_2^2Ad^2\psi}_{A,B} < \frac{1}{m^2L}.
\]
Here for $x = x_1A + x_2B$ with $x_1, x_2 \in \RR$ we define 
\[
\norm{x}_{A,B} := \max\{\norm{x_1}_{\RR/\ZZ}, \norm{x_2}_{\RR/\ZZ}\},
\]
where $\norm{x_j}_{\RR/\ZZ}$ denotes the distance from $x_j \in \RR$ to the closest integer.
\end{lemma}

\begin{proof}
Notice that we can identify the fundamental domain of $\CC/(A\ZZ+B\ZZ)$ with $A[0,1)+B[0,1)$ and think of each $I_j$ as a ``box'' in the ``directions'' $A$ and $B$ with ``side lengths'' $1/(m^2L)$.

We have $m_2m_3x + m_1m_4y + m_1m_2z = (m_2m_3+m_1m_4+m_1m_2)w = mw$ by applying \cref{lem:alg-id}. Let $W, X, Y, Z\in\Lambda^2$ be such that $w\in I_{W}$, $x\in I_{X}$, $y\in I_{Y}$, and $z\in I_{Z}$. We will write $W=(W_1,W_2)$ and similar for $X,Y,Z$. Then $m_2m_3x+m_1m_4y+m_1m_2z \pmod{A\ZZ+B\ZZ}$ lies in
\begin{align*}
&A\left[\frac{m_2m_3X_1+m_1m_4Y_1+m_1m_2Z_1}{mL}, \frac{m_2m_3X_1+m_1m_4Y_1+m_1m_2Z_1}{mL} + \frac{1}{mL}\right)\\
&+B\left[\frac{m_2m_3X_2+m_1m_4Y_2+m_1m_2Z_2}{mL}, \frac{m_2m_3X_2+m_1m_4Y_2+m_1m_2Z_2}{mL} + \frac{1}{mL}\right)
\end{align*}
and $mw \pmod{A\ZZ+B\ZZ}$ lies in
\[A\left[\frac{mW_1}{mL}, \frac{mW_1}{mL} + \frac{1}{mL}\right)+B\left[\frac{mW_2}{mL}, \frac{mW_2}{mL} + \frac{1}{mL}\right).\]
Since $m_2m_3X_j+m_1m_4Y_j+m_1m_2Z_j < mL$, these two boxes intersect exactly when
\[m_2m_3X+m_1m_4Y+m_1m_2Z=mW\]
as ordered pairs, which implies that $W=X=Y=Z$ since $\Lambda$ and hence $\Lambda^2$ has no nontrivial solutions to this equation. The conclusion follows from the fact that $w$ and $x$ lie in the box $I_w$ with side lengths $1/(m^2L)$ and from
\[x-w=2m_1m_2(m_2An_1 + m_1Bn_2)d\psi + m_1^2m_2^2Ad^2\psi,\]
which is verified by expanding the definitions of $w,x$. 
\end{proof}

Finally, following \cite{FSSSZ}, we need irrational numbers well-approximable by fractions with a special property.

\begin{lemma}[{\cite[Lemma~3.3]{FSSSZ}}] \label{lem:alpha-hard}
Fix a positive integer $m > 1$. Then there is a real $b \in (1, 2^{2m + 1}]$ such that the following holds. For all real $r > 0$, there is an irrational number $\psi$ and infinitely many fractions $p_i/q_i$ with relatively prime positive integers $p_i < q_i$ and $q_i$ having no prime factor smaller than $m$ such that $|\psi - p_i/q_i| < 1/(mq_i^2)$, and $rb^i < q_i < 2rb^i$ for $i\ge i(r, m, b)$ sufficiently large.
\end{lemma}

We are ready to prove \cref{thm:2d-general}.
\begin{proof}[Proof of \cref{thm:2d-general}]
We may assume that $\alpha$ is sufficiently small or otherwise we can take $A = [N/2]\times [N]$ and then the theorem is true if the constant is chosen appropriately.

Let $L = \exp(c \log(1/\alpha)^2)$ for an appropriately chosen sufficiently small constant $c > 0$. Apply \cref{lem:alpha-hard} for $m = L$ and $t = 2L + 1$ different values of $r$, namely $r = 2^j$ for $1\le j\le 2L + 1$. The lemma gives a single $b \in (1, 2^{2L + 1}]$ and irrationals $\psi_1, \ldots, \psi_t$ as well as positive integers $p_{j, i}, q_{j, i}$ with $\gcd(p_{j, i}, q_{j, i}) = 1$ so that for all $j \in [t]$,
\begin{itemize}
    \item $q_{j, i}\in (2^jb^i, 2^{j + 1}b^i)$ for sufficiently large $i\ge i(j)$, and
    \item $\gcd(q_{j, i}, \text{lcm}(1, \ldots, L)) = 1$ for $i\ge i(j)$, and
    \item $|\psi_j - \frac{p_{j, i}}{q_{j, i}}| < \frac{1}{Lq_{j, i}^2}$ for $i\ge i(j)$.
\end{itemize}
Let $I = \max\{i(1), \ldots, i(t)\}$. Then the above properties hold for all $1\le j\le t$ and $i\ge I$. Observe that all sufficiently large $N$ (here ``sufficiently large'' depends on $\alpha$) are within a factor of $4$ from some $q_{j, i}$ with $1\le j\le t$ and $i\ge I$. Therefore, to prove the theorem for all sufficiently large integers $N$, it suffices to prove it for numbers of the form $N = q_{j, i}$.

Let $N = q_{j, i}$ with $1\le j\le t$ and $i\ge I$. Let $\psi = \psi_j$. Define
\[\mc{A} = \{(n_1,n_2)\in [N]^2: \psi f(n_1,n_2)\in \mc{B}\pmod{A\ZZ+B\ZZ}\},\]
where $f$ is defined via \cref{eq:general-f-def},
\[
\mc{B} = \bigcup_{(k_1,k_2) \in \Lambda^2}\left(A\left[\frac{k_1}{mL}, \frac{k_1}{mL} + \frac{1}{m^2L}\right) + B\left[\frac{k_2}{mL}, \frac{k_2}{mL} + \frac{1}{m^2L}\right)\right)\subseteq\CC/(A\ZZ+B\ZZ),
\]
and $\Lambda$ is a subset of $\{0, 1,\dots, L-1\}$ of size $Le^{- O(\sqrt{\log L})}$ not containing nontrivial solutions to $m_2m_3x+m_1m_4y+m_1m_2z=mw$ (by \cref{lem:general-behrend-modified}). By the Weyl equidistribution\footnote{
Let us check the equidistribution more carefully. We have the identity
\[\psi f(n_1, n_2) = A(m_2^2n_1^2)\psi + B(2m_1m_2n_1n_2)\psi + \frac{B^2}{A}m_1^2n_2^2 = A(m_2^2n_1^2 + am_1^2n_2^2)\psi + B(2m_1m_2n_1n_2 + bm_1^2n_2^2)\psi\]
if we uniquely write $B^2/A = aA + bB$ for $a, b\in\RR$. Thus checking equidistribution in $\CC/(A\ZZ+B\ZZ)$ is equivalent to checking equidistribution of $(n_1, n_2)\mapsto (m_2^2\psi, 0)n_1^2 + (0, 2m_1m_2\psi)n_1n_2 + (am_1^2\psi, bm_1^2\psi)n_2^2$ in $(\RR/\ZZ)^2$. This does indeed follow from \cite[Exercise~1.1.6]{Tao12}, in fact regardless of what $a, b\in\RR$ are.} criterion (e.g., see \cite{Tao12}), using $m_{A,B}(\cdot)$ for Lebesgue measure normalized so that $A[0,1)+B[0,1)$ has measure $1$, we see that as $N \to \infty$,
\[
\frac{|\mc{A}|}{N^2} \to m(B) = \frac{|\Lambda|^2}{(m^2L)^2} = e^{-O(\sqrt{\log L})} \ge 2\alpha 
\] 
as long as we have chosen the constant $c$ in $L = \exp(c \log(1/\alpha)^2)$ so that the last inequality is true. Thus, for sufficiently large $N$, we have $|\mc{A}|\ge\alpha N^2$.

A key point here is that while the rate of convergence of the equidistribution claim may depend on $\psi$, since there are only finitely many $\psi$'s that we need to consider, there is a single $N_0(\alpha)$ such that $|\mc{A}| \ge \alpha N^3$ whenever $N = q_{j,i} \ge N_0(\alpha)$ with $j \in [t]$ and $i \ge I$ as above.

Fix a nonzero integer $s$ with $|s| < N$. Suppose $(a_1, a_2)$ satisfies
\[
(a_1, a_2), (a_1 + m_1s, a_2), (a_1, a_2 + m_2s), (a_1 - m_3s, a_2 - m_4s) \in\mc{A}.
\]
By \cref{lem:general-alpha-transfer}, $\norm{2m_1m_2(m_2Aa_1 + m_1Ba_2)s\psi + m_1^2m_2^2As^2\psi}_{A,B} < 1/(m^2L)$. So
\begin{align*}
\norm{2m_1m_2(m_2Aa_1 + m_1Ba_2)s\frac{p_{j, i}}{q_{j, i}} + m_1^2m_2^2As^2\psi}_{A,B}
&\le \frac{1}{m^2L} + 2m_1m_2s\abs{m_2Aa_1 + m_1Ba_2}\abs{\psi - \frac{p_{j, i}}{q_{j, i}}} 
\\
&\le \frac{1}{m^2L} + 4m^2(|A| + |B|)N^2 \cdot \frac{1}{Lq_{j, i}^2}\\
&= O\left(\frac{1}{L}\right).
\end{align*}

Recall that $N = q_{j,i}$ is relatively prime to all of $[L]$ as well as to $p_{j,i}$. In particular, $N$ is odd and provided we chose $L$ large enough, it is relatively prime to $m_1m_2$ as well. Also $|s| < N$, so $s$ is not divisible by $N$. It follows that $2m_1m_2sp_{j, i}/q_{j, i}$ is not an integer. Writing $2m_1m_2sp_{j, i}/q_{j, i} = P/Q$ where $P$ and $Q$ are relatively prime integers with $Q$ positive, one has $Q > L$ since all prime divisors of $q_{j,i}$ are greater than $L$.

Thus $\norm{(m_2Aa_1+m_1Ba_2) P/Q + m_1^2m_2^2As^2\psi}_{A,B} = O(1/L)$. Since multiplication by $P$ is a bijection in $\ZZ/Q\ZZ$, we see there are at most $(1 + O(Q/L))^2 = O(Q^2/L^2)$ possible values that $(a_1,a_2)$ can take in $(\ZZ/Q\ZZ)^2$, and hence there are $O(N^2/L^2)$ possible values (recall $N/Q \in \ZZ$) that $(a_1, a_2)$ can take in $[N]^2$. Therefore there are $O(N^2/L^2) = O(e^{-2c\log(1/\alpha)^2}N^2)$ different points $(a_1, a_2) \in [N]^2$ that generate a pattern of common difference $s$.
\end{proof}
Now we are ready to prove \cref{thm:4pt-2dim}.
\begin{proof}[Proof of \cref{thm:4pt-2dim}]
If the pattern $P$ is strictly nonconvex, use \cref{thm:2d-general}. If the pattern $P$ is strictly convex, use \cref{thm:sec3-main}. If the pattern contains three collinear points, using the trivial observation that if $P'\subseteq P$ then $\pdd_P(\alpha)\le\pdd_{P'}(\alpha)$ and using \cref{thm:corner-new} proves the result.
\end{proof}

\section{Special four-point patterns in one dimension} \label{sec:4pt-1d-special}
In this section we prove that for any $C > 0$, certain $4$-point patterns $P$ on the line have $\pdd_P(\alpha)<\alpha^C$ for all sufficiently small $\alpha > 0$. The following theorem immediately implies \cref{thm:4pt-special-1dim}.
\begin{theorem}\label{thm:1d-special}
For any $C>0$, there exist $\alpha_0\in(0,1)$ and $a_j\in\NN$ such that the following holds. Let $0 < \alpha < \alpha_0$ and $P = \{0,a_1,a_2, a_3\}$. For all sufficiently large $N$, there exists $A \subseteq [N]$ with $|A| \ge \alpha N$ such that for all nonzero integers $d$, there are at most $\alpha^CN$ points $x \in \ZZ$ such that $x + d\cdot P := \{x + dt : t \in P\} \subseteq A$.
\end{theorem}
For the rest of the section, let $\omega = \exp(\pi i/6)$. We first need the following well-known number theoretic fact.
\begin{proposition}\label{prop:chebotarev}
Let $K = \QQ(\omega, 3^{1/6})$. Then a $1/24$ density of primes split completely over $K$.
\end{proposition}
\begin{proof}
This is a direct consequence of Chebotarev's density theorem applied to $K$. See \cite{LO77} for an effective version.
\end{proof}
Let
\begin{align*}
P_1(X,Y,Z)&=(X-Y)(Y-Z)(Z-X), \\
P_2(X,Y,Z) &=YZ(Y-Z),\\
P_3(X,Y,Z)&=ZX(Z-X), \text{ and} \\
P_4(X,Y,Z) &=XY(X-Y).
\end{align*}
They are chosen to satisfy the polynomial identity
\begin{equation}\label{eq:section-5-identity}
P_1(X,Y,Z)T^2+P_2(X,Y,Z)(T+XD)^2+P_3(X,Y,Z)(T+YD)^2+P_4(X,Y,Z)(T+ZD)^2=0.
\end{equation}
Define the constants
\[z_1 = 3^{-1/6}\omega^5,\qquad z_2 = \frac{1}{2}(3^{1/3}\omega^2+3^{5/6}\omega^5),\qquad z_3 = \frac{1}{2}(-3^{1/3}\omega^2+3^{5/6}\omega^5).\]
They satisfy the relations
\begin{align*}
P_1(z_1,z_2,z_3) &= -1,\\
P_2(z_1,z_2,z_3) &= \phantom{-}3,\\
P_3(z_1,z_2,z_3) &= -1, \text{ and}\\
P_4(z_1,z_2,z_3) &= -1,
\end{align*}
Let $p$ be a large prime, to be chosen later, such that $p$ splits completely over $K = \QQ(\omega, 3^{1/6})$. By \cref{prop:chebotarev}, such primes exist. Because $p$ splits completely, we find that there exist integers $a_1,a_2,a_3\in[p]$ satisfying
\begin{equation} \label{eq:omega-solutions}
\begin{split}
P_1(a_1,a_2,a_3) &\equiv -1 \pmod{p},\\
P_2(a_1,a_2,a_3) &\equiv \phantom{-}3 \pmod{p},\\
P_3(a_1,a_2,a_3) &\equiv -1 \pmod{p},\\
P_4(a_1,a_2,a_3) &\equiv -1 \pmod{p},
\end{split}
\end{equation}
namely, by replacing $z_1,z_2,z_3$ with their reductions in $\mc{O}_K/\mathfrak{p}\simeq\FF_p$, where $\mc{O}_K$ is the ring of integers of $K$ and $\mathfrak{p}$ is any prime ideal of $\mc{O}_K$ lying over $p$. Although $z_1\notin\mc{O}_K$, we see that $3z_1\in\mc{O}_K$, so as long as $p > 3$ this is still valid. 
For the remainder of this section, let $P=\{0,a_1,a_2,a_3\}$.

\begin{lemma}\label{lem:chebotarev-behrend-tensor}
There is an absolute constant $c > 0$ such that the following holds. For all $p,a_1,a_2,a_3$ as above, and for sufficiently large $L$ in terms of $p$, there exists $\Lambda\subseteq\{0,\ldots,L-1\}$ avoiding nontrivial solutions to 
\[
P_1(a_1,a_2,a_3) w + P_2(a_1,a_2,a_3)x+P_3(a_1,a_2,a_3)y +P_4(a_1,a_2,a_3)z = 0
\]
with at least $L^{1-c/\sqrt{\log p}}$ elements. Here a nontrivial solution is one where not all $x,y,z,w$ are equal.
\end{lemma}

\begin{proof}
Let $S$ be a subset of $[p]$ of size at least $p\exp(-c'\sqrt{\log p})$ which avoids nontrivial solutions to $-w+3x-y-z\equiv 0\pmod{p}$, constructed via a standard modification of Behrend's construction \cite{B}. Note that by \cref{eq:omega-solutions} it also avoids nontrivial solutions to
\[P_1(a_1,a_2,a_3)w + P_2(a_1,a_2,a_3)x + P_3(a_1,a_2,a_3)y + P_4(a_1,a_2,a_3)z\equiv 0\pmod{p},\]
i.e., has no solutions other than $w=x=y=z$. Now consider 
\[T_n = \{x\in\{0,\ldots,p^n-1\}: \text{all digits base }p\text{ are in }S\},\]
where $n = \lfloor\log_p L\rfloor$. It is easy to verify that $T_n$ avoids nontrivial solutions to
\[P_1(a_1,a_2,a_3)w + P_2(a_1,a_2,a_3)x + P_3(a_1,a_2,a_3)y + P_4(a_1,a_2,a_3)z = 0.\]
Indeed, reducing the equation mod $p$ and comparing the last digits of $w,x,y,z$ in base $p$, we find that those digits must all be the same (as they are in $S$, which avoids nontrivial solutions mod $p$). Then we can subtract off those final digits and divide by $p$, and hence repeat the argument. It is easy to see that $\Lambda = T_n$ is a set with the desired property and size, as long as $L$ is large enough in terms of $p$.
\end{proof}
Note this trick of reducing mod $p$ to reduce to the equation $x+y+z=3w$ and using base expansion is in the proof of \cite[Theorem~7.5]{Ruz2}. Now we proceed to the next step, embedding our subset into the torus $\RR/\ZZ$.
\begin{lemma}\label{lem:four-point-alpha-transfer}
For all $p,a_1,a_2,a_3$ as above, there exist $\Theta_1,\Theta_2,\Theta_3>0$ such that the following holds. Let $\Lambda$ be a subset of $\{0, 1, \ldots, L - 1\}$ not containing any nontrivial solutions to
\[P_1(a_1,a_2,a_3)w + P_2(a_1,a_2,a_3)x+P_3(a_1,a_2,a_3)y+P_4(a_1,a_2,a_3)z = 0\]
and let $\psi$ be a fixed real constant. For each $j\in\Lambda$, let
\[I_j := \left[\frac{j}{\Theta_1L}, \frac{j}{\Theta_1L} + \frac{1}{\Theta_1^2L}\right) = \RR/\ZZ,
\]
and let
\[
B = \bigcup_{j \in \Lambda} I_j.
\]
Let $n,d\in\ZZ$ and $w = \psi n^2$, $x = \psi (n+a_1d)^2$, $y = \psi (n+a_2d)^2$, and $z = \psi (n+a_3d)^2$. Suppose that $w, x, y, z \pmod 1$ all lie in $B$. Then
\[
\norm{\Theta_2\psi nd}_{\RR/\ZZ} < \frac{\Theta_3}{L},
\]
where $\norm{x}_{\RR/\ZZ}$ denotes the distance from $x \in \RR$ to the closest integer.
\end{lemma}
\begin{proof}
The proof is similar to the proof of \cref{lem:general-alpha-transfer}, and also similar to the proof of \cite[Lemma~3.7]{FSSSZ}. By the polynomial identity \cref{eq:section-5-identity}, we have
\[P_1(a_1,a_2,a_3)w + P_2(a_1,a_2,a_3)x+P_3(a_1,a_2,a_3)y+P_4(a_1,a_2,a_3)z = 0.\]
Hence if $w\in I_W, x\in I_X, y\in I_Y, z\in I_Z$ we deduce
\begin{align*}
\norm{\frac{P_1(a_1,a_2,a_3)W + P_2(a_1,a_2,a_3)X+P_3(a_1,a_2,a_3)Y+P_4(a_1,a_2,a_3)Z}{\Theta_1L}}_{\RR/\ZZ}&<\frac{\sum_{j=1}^4|P_j(a_1,a_2,a_3)|}{\Theta_1^2L}\\
&\le\frac{1}{\Theta_1L}
\end{align*}
if $\Theta_1$ is chosen sufficiently large. This implies
\[P_1(a_1,a_2,a_3)W + P_2(a_1,a_2,a_3)X+P_3(a_1,a_2,a_3)Y+P_4(a_1,a_2,a_3)Z = 0,\]
but since $W,X,Y,Z\in\Lambda$ we deduce that $W = X = Y = Z$. Finally, we find $\psi n^2,\psi(a+a_jd)^2$ are all in an interval of length $1/(\Theta_1^2L)$. Thus
\[(a_2^2-a_3^2)\psi(n+a_1d)^2+(a_3^2-a_1^2)\psi(n+a_2d)^2+(a_1^2-a_2^2)\psi(n+a_3d)^2 = 2(a_1-a_2)(a_2-a_3)(a_3-a_1)nd\]
is small. The result follows upon choosing $\Theta_2 = |2(a_1-a_2)(a_2-a_3)(a_3-a_4)|$ and $\Theta_3$ appropriately.
\end{proof}
We now ready to prove \cref{thm:1d-special}.
\begin{proof}[Proof of \cref{thm:1d-special}]
We may assume that $\alpha$ is sufficiently small because we can choose $\alpha_0$ appropriately. Choose $p,a_1,a_2,a_3$ as at the beginning of the section; we will later ensure that $p$ is sufficiently large in terms of $C$. Additionally, $\{0,a_1,a_2,a_3\}$ will be the pattern we are considering. Also, $\Theta_1,\Theta_2,\Theta_3$ will be chosen as in \cref{lem:four-point-alpha-transfer} (note they depend only on $p,a_1,a_2,a_3$).

Let $L = \alpha^{-c'\sqrt{p}}$ for an appropriately chosen sufficiently small absolute constant $c' > 0$. Apply \cref{lem:alpha-hard} for $m = L$ and $t = 2L + 1$ different values of $r$, namely $r = 2^j$ for $1\le j\le 2L + 1$. The lemma gives a single $b \in (1, 2^{2L + 1}]$ and irrationals $\psi_1, \ldots, \psi_t$ as well as positive integers $p_{j, i}, q_{j, i}$ with $\gcd(p_{j, i}, q_{j, i}) = 1$ so that for all $j \in [t]$,
\begin{itemize}
    \item $q_{j, i}\in (2^jb^i, 2^{j + 1}b^i)$ for sufficiently large $i\ge i(j)$, and
    \item $\gcd(q_{j, i}, \text{lcm}(1, \ldots, L)) = 1$ for $i\ge i(j)$, and
    \item $|\psi_j - \frac{p_{j, i}}{q_{j, i}}| < \frac{1}{Lq_{j, i}^2}$ for $i\ge i(j)$.
\end{itemize}
Let $I = \max\{i(1), \ldots, i(t)\}$. Then the above properties hold for all $1\le j\le t$ and $i\ge I$. Observe that all sufficiently large $N$ (here ``sufficiently large'' depends on $\alpha$) are within a factor of $4$ from some $q_{j, i}$ with $1\le j\le t$ and $i\ge I$. Therefore, to prove the theorem for all sufficiently large integers $N$, it suffices to prove it for numbers of the form $N = q_{j, i}$.

Let $N = q_{j, i}$ with $1\le j\le t$ and $i\ge I$. Let $\psi = \psi_j$. Define
\[F = \{n\in\mathbb{N}: n\psi\in B\pmod{1}\},\]
where, as in \cref{lem:four-point-alpha-transfer},
\[
\mc B = \bigcup_{k \in \Lambda} \left[\frac{k}{\Theta_1L}, \frac{k}{\Theta_1L} + \frac{1}{\Theta_1^2L}\right) \subseteq \RR/\ZZ
\]
and $\Lambda$ is a subset of $\{0, 1,\dots, L-1\}$ of size $L^{1-c/\sqrt{\log p}}$ (by \cref{lem:chebotarev-behrend-tensor}) not containing nontrivial solutions to
\[P_1(a_1,a_2,a_3)w + P_2(a_1,a_2,a_3)x+P_3(a_1,a_2,a_3)y+P_4(a_1,a_2,a_3)z = 0.\]
Here $\Theta_1,\Theta_2,\Theta_3$ are taken to depend on $p,a_1,a_2,a_3$ as in \cref{lem:four-point-alpha-transfer}. Let
\begin{equation}
\mc A = \{x\in [N]: x^2  \in F\}.
	\label{eq:construction-A}
\end{equation}
By the Weyl equidistribution criterion (e.g., see \cite{Tao12}), using $m(\cdot)$ for Lebesgue measure, as $N \to \infty$,
\[
\frac{|\mc A|}{N} \to m(\mc B) = \frac{|\Lambda|}{\Theta_1^2L} = \frac{1}{\Theta_1^2}L^{-c/\sqrt{p}} \ge 2\alpha 
\]
as long as we have chosen the constant $c'$ in $L = \alpha^{-c'\sqrt{p}}$ so that the last inequality is true for $\alpha < \alpha_0$. (Thus $\alpha_0$ will be chosen in terms of $p$ and therefore ultimately $C$). Thus, for sufficiently large $N$, we have $|A|\ge\alpha N^3$.

A key point here is that while the rate of convergence of the equidistribution claim may depend on $\psi$, since there are only finitely many $\psi$'s that we need to consider, there is a single $N_0(\alpha)$ such that $|\mc A| \ge \alpha N^3$ whenever $N = q_{j,i} \ge N_0(\alpha)$ with $j \in [t]$ and $i \ge I$ as above.

Fix a nonzero integer $s$ with $|s| < N$. Suppose $a$ satisfies
\[
a,a+a_1s,a+a_2s,a+a_3s \in \mc A.
\]
Then 
\[
a^2,(a+a_1s)^2,(a+a_2s)^2,(a+a_3s)^2 \in F
\]
by the construction~\cref{eq:construction-A}. By \cref{lem:four-point-alpha-transfer}, $\norm{\Theta_2sa\psi}_{\RR/\ZZ} < \Theta_3/L$. So
\begin{align*}
\norm{\Theta_2sa\frac{p_{j, i}}{q_{j, i}}}_{\RR/\ZZ}
&\le \norm{\Theta_2sa\psi}_{\RR/\ZZ} + 
\abs{\Theta_2sa\psi - \Theta_2sa\frac{p_{j, i}}{q_{j, i}}} 
\\
&\le \frac{\Theta_3}{L} + \Theta_2s\abs{a}\abs{\psi - \frac{p_{j, i}}{q_{j, i}}} 
\\
&\le \frac{\Theta_3}{L} + \Theta_2N^2 \cdot \frac{1}{Lq_{j, i}^2} = \frac{\Theta_3}{L} + \frac{\Theta_2}{L}.
\end{align*}

Recall that $N = q_{j,i}$ is relatively prime to all of $[L]$ as well as to $p_{j,i}$. In particular, as long as $L$ is big enough (which we can guarantee), we have $\gcd(N,\Theta_2) = 1$. Also $|s| < N$, so $s$ is not divisible by $N$. It follows that $\Theta_2sp_{j, i}/q_{j, i}$ is not an integer. Writing $\Theta_2sp_{j, i}/q_{j, i} = P/Q$ where $P$ and $Q$ are relatively prime integers with $Q$ positive, one has $Q > L$ since all prime divisors of $q_{j,i}$ are greater than $L$.

Thus $\norm{aP/Q}_{\RR/\ZZ} \le (\Theta_2+\Theta_3)/L$. So $aP \pmod Q \in [-\floor{(\Theta_2+\Theta_3)Q/L},\floor{(\Theta_2+\Theta_3)Q/L}]$. Since multiplication by $P$ is a bijection in $\ZZ/Q\ZZ$, there are at most $1 + (2\Theta_2+2\Theta_3)Q/L \le (1+2\Theta_2+2\Theta_3)Q/L$ possible values that $a$ can take in $\ZZ/Q\ZZ$, and hence there are at most $(1+2\Theta_2+2\Theta_3)N/L$ possible values (recall $N/Q \in\ZZ$) that $a$ can take in $[0,N)$. Therefore there are at most $DN/L \le D\alpha^{c'\sqrt{p}}N$ different points $a\in [N]$ that generate a pattern $\{0,a_1,a_2,a_3\}$ of difference $s$, where $D = 2+4\Theta_2+4\Theta_3$.

Now as long as we choose $p$ such that $c'\sqrt{p} > C$ and $p$ splits completely in $K = \QQ(\omega, 3^{1/6})$, the construction achieves the desired bounds.
\end{proof}
\begin{remark}
Gowers~\cite{GowersDoc} (also see \cite{GreenOpen}) asked  whether every Fourier-uniform subsets of $\ZZ/N\ZZ$ with density $\alpha$ contains at least $(\alpha^{1000} - o(1)) N^2$ 4-term arithmetic progressions. He constructed a counterexample to an earlier conjecture \cite[Conjecture 4.1]{Gow01} that every Fourier-uniform subset of $\ZZ/N\ZZ$ of density $\alpha$ contains at least $(\alpha^4 - o(1)) N^2$ 4-APs . Here $A$ is said to be Fourier-uniform if $\sup_{r \in [N-1]} \sum_{j \in A} e^{2\pi ijr /N} = o(N)$.

The construction just given for \cref{thm:1d-special} demonstrates that for every $C>0$ there exists some 4-point pattern $P \subseteq \ZZ$ such that for all sufficiently small $\alpha$ there exists Fourier-uniform subsets of $\ZZ/N\ZZ$ of density $\alpha + o(1)$ that contains at most $\alpha^C N^2$ copies of the pattern $P$ (allowing translations and dilations). The Fourier-uniformity of this construction can be verified by standard exponential sum estimates via Weyl's inequality.

Furthermore, if there exists a subset of $[N]$ of size $N^{1-o(1)}$ avoiding nontrivial solutions to $x + 8y = 3z + 6w$ (it is an open problem whether such sets exist), then a modification of the construction would produce some $A \subseteq [N]$ with density $\alpha + o(1)$ such that contains at most $\alpha^{\omega(1)}N^2$ translated dilates of the $4$-point patterns $P = \{0,1,2,4\}$. This set $A$ has the additional property that for every nonzero $d$, it contains at most $\alpha^{\omega(1)}N$ translates of $d \cdot P$. In contrast, for $P = \{0,1,2,3\}$, no such $A$ can exist due to \cref{thm:4ap}.
\end{remark}

\section{Four-point patterns in one dimension} \label{sec:4pt-1d}

In this section, we prove \cref{thm:4pt-1dim}. We begin with an easy special case that illustrates our constructions.

\begin{proposition}\label{prop:0125}
There exists some constant $c > 0$ so that for all $0 < \alpha < 1/2$ and all sufficiently large prime $N > N_0(\alpha)$, 
there exists some $f \colon \ZZ/N\ZZ \to [0,1]$ with $\EE f \ge \alpha$ so that for every $d \ne 0$
\[
\EE_t f(t) f(t+d)f(t+2d)f(t+5d) < (1-c)\alpha^4.
\]	
\end{proposition}

\begin{proof}
    Let $a_1=-6$, $a_2=15$, $a_3=-10$, $a_4=1$.
    Let $\omega = \exp(2\pi i/N)$ and set, for some $\gamma_1, \gamma_2, \gamma_3, \gamma_4 \in [-1/8,1/8]$,
	\[
	f(t) = \alpha \paren{ 1 +  \sum_{k=1}^{4} 2\gamma_k\cos \paren{\frac{2\pi a_k t^2}{N}}}
	=
	\alpha \paren{ 1 + \sum_{k=1}^4 \gamma_k (\omega^{a_k t^2} + \omega^{-a_k t^2})}.
	\]
	Applying the Gauss sum estimate
	\[
	\abs{\sum_{t \in [N]} \omega^{\ell t^2}} \le \frac{1}{\sqrt{N}} \qquad \text{for all nonzero } \ell \in \ZZ/N\ZZ,
	\]
	we obtain
	\[
	\EE f = \alpha \paren{ 1+ O(N^{-1/2}) }.
	\]
	By expanding, we obtain, uniformly for every $d \ne 0$,
	\begin{equation} \label{eq:0125-expand}
	\EE_t f(t) f(t+d)f(t+2d)f(t+5d)	
	= \alpha^4 \paren{1 + 2\gamma_1\gamma_2\gamma_3\gamma_4 + O(N^{-1/2}) }
	\end{equation}
	since the only choices $b_1, b_2, b_3, b_4 \in \{0, \pm 1, \pm 6, \pm 10, \pm 15\}$ such that
	\[
	b_1 t^2 + b_2 (t+d)^2 + b_3 (t+2d)^2 + b_4 (t+5d)^2
	\]
	does not depend on $t$ are exactly $(b_1, b_2, b_3, b_4) = (0,0,0,0)$ and $\pm (6, -15, 10, -1)$, for which the sum is identically zero. The remaining terms in the expansion, after averaging over $x$, are $O(N^{-1/2})$ by Gauss sum estimates. Now choosing $\gamma_1 = \gamma_2 = \gamma_3 = -\gamma_4 = 1/8$ yields the result.
\end{proof}

The above proof is simpler than the general case, where we may see additional significant terms on the right hand side of the expression corresponding to \cref{eq:0125-expand}. 
For example, when $P = \{0,1,2,4\}$, we take
\[
f(t) = 
	\alpha \paren{ 1 + \sum_{k=1}^4 \gamma_k (\omega^{a_k t^2} + \omega^{-a_k t^2})}
\]
with $a_1 = -3$, $a_2 = 8$, $a_3 = -6$, $a_4 = 1$ chosen to satisfy the polynomial identity (in $t$ and $d$)
\[
a_1 t^2 + a_2 (t+d)^2 + a_3 (t+2d)^2 + a_4 (t+4d)^2 = 0.
\]
However, unlike the pattern $P = \{0,1,2,5\}$ from \cref{prop:0125}, one has additional relations (which we will call ``degeneracies''):
\begin{align*}
6 t^2 - 6(t+d)^2 - 3(t+2d)^2 + 3(t+4d)^2 &=  30d^2, \\
3 t^2 - 6(t+d)^2 + 3(t+2d)^2 &= 6d^2, \\
3 t^2            - 6(t+2d)^2 + 3(t+4d)^2 &=  24d^2.
\end{align*}
Using the above relations (and it turns out that these are the only ones), and applying the Gauss sum estimate, we find that, uniformly for all nonzero $d$,
\begin{multline*}
\EE_x f(x) f(x+d)f(x+2d)f(x+4d) \\ 
= \alpha^4 (1 + 2\gamma_1 \gamma_2 \gamma_3 \gamma_4 + \gamma_1^2 \gamma_3^2 (\omega^{30d^2} + \omega^{-30d^2}) + \gamma_1^2 \gamma_3 (\omega^{24d^2} +
\omega^{6d^2} + \omega^{-6d^2} + \omega^{-24d^2}) + O(N^{-\frac{1}{2}}))
\end{multline*}
By setting $\gamma_2 = \gamma_3 = \gamma_4 = 1/8$ and $\gamma_1 = -1/512$, we find that $\sup_{d \ne 0} \EE_x f(x) f(x+d)f(x+2d)f(x+4d) \le (1-c - o(1)) \alpha^4$ for some constant $c > 0$.

In the remainder of the section we establish the following claim, 
which implies \cref{thm:4pt-1dim} by a standard probabilistic argument 
where we use $f$ to sample a random set $A \subseteq \ZZ/N\ZZ$ so that $x$ is included in $A$ with probability $f(x)$ independently for all $x \in \ZZ/N\ZZ$. 
A standard concentration argument, e.g., via the bounded difference inequality, then implies that $A$ satisfies the desired conclusion of \cref{thm:4pt-1dim} with positive probability.
Note that by changing $N$ to $N + o(N)$ if necessary, we may assume, for the purpose of proving \cref{thm:4pt-1dim}, that $N$ is prime for the rest of this section.

\begin{theorem}\label{thm:1d-general}
There exists some constant $c > 0$ so that for all positive integers $k_1 <k_2<k_3$ with $k_3 \ne k_1+ k_2$, all $0 < \alpha < 1/2$, and all sufficiently large prime $N > N_0(\alpha, k_i)$, there exists some $f \colon \ZZ/N\ZZ \to [0,1]$ with $\EE f \ge \alpha$ so that for every $d \ne 0$
\[
\EE_t f(t) f(t+k_1d)f(t+k_2d)f(t+k_3d) < (1-c)\alpha^4.
\]
\end{theorem}

We may assume that $k_1 + k_2 < k_3$. Indeed, if $k_3 > k_1 + k_2$, then by a change of variable from $d$ to $-d$, the problem is equivalent to the pattern $\{0, k_3-k_2, k_3-k_1, k_3\}$.

We reparametrize by defining positive integers
\begin{equation}\label{eq:xyz}
x = k_1,\qquad y = k_2 - k_1,\qquad z = k_3 - k_1 - k_2,
\end{equation}
so that
\begin{equation}\label{eq:k1k2k3}
k_1 = x,\qquad k_2 = x + y,\qquad k_3 = 2x + y + z.
\end{equation}
Now define
\begin{align}\label{eq:a1a2a3a4}
\begin{split}
a_1 &= - y (x + z)(x + y + z),\\
a_2 &= (x + y)(x+z)(2x + y + z),\\
a_3 &= -x (2x + y + z)(x + y + z),\\
a_4 &= x y(x + y),\\
\end{split}
\end{align}
which are defined so that the following polynomial identity holds with indeterminates $T$ and $D$:
\begin{equation}  \label{eq:a-quadratic-identity}
a_1 T^2 + a_2 (T + k_1 D)^2 + a_3 (T + k_2 D)^2 + a_4 (T + k_3 D)^2 = 0.
\end{equation}

We write
\[
a_0 = 0 \quad \text{and} \quad
a_j = -a_{-j} \quad \text{ for  }j\in\{1,2,3,4\}.
\]
For the construction, similar to the example above, we set
\begin{equation}\label{eq:naive-f}
f(t) = \alpha \paren{1 + \sum_{k=1}^4 2\gamma_k\cos \paren{\frac{2\pi a_k t^2}{N}}} = \alpha \paren{ 1 + \sum_{k=1}^4 \gamma_k(\omega^{a_kt^2} + \omega^{-a_kt^2})},
\end{equation}
where $\omega = \exp(2\pi i/N)$ and $\gamma_k\in[-1/8,1/8]$ are real parameters that we will choose later. 
Note that we will also use the convention that 
\[
\gamma_{0}=1 \quad \text{and} \quad \gamma_{-k} = \gamma_k \text{ for } k=1,2,3,4.
\]

A \emph{signature} is a tuple $(i_1,i_2,i_3,i_4)$ of integers with $-4\le i_1,i_2,i_3,i_4\le 4$. Define
\[
u_{i_1,i_2,i_3,i_4}(t,d) = \omega^{a_{i_1}t^2+a_{i_2}(t+k_1d)^2+a_{i_3}(t+k_2d)^2+a_{i_4}(t+k_3d)^2}.\]
Define polynomials $p_1^I(X,Y,Z)$, $p_2^I(X,Y,Z)$, and $p_3^I(X,Y,Z)$ in the variables $X,Y,Z$ so that 
\begin{equation}\label{eq:p1p2p3}
a_{i_1}T^2+a_{i_2}(T + k_1 D)^2+a_{i_3}(T +k_2 D)^2+a_{i_4}(T+k_3D)^2
= p_1^I(x,y,z)T^2+p_2^I(x,y,z)TD + p_3^I(x,y,z)D^2
\end{equation}
as polynomials in $T$ and $D$, for any choice of $x,y,z$.
In other words, we substitute $a_1,a_2,a_3,a_4,k_1,k_2,k_3$ for polynomials in $x,y,z$ according to \cref{eq:k1k2k3} and \cref{eq:a1a2a3a4} to write the left-hand side as a polynomial in $x,y,z,T,D$, and then collect the coefficients of $T^2$, $TD$, and $D^2$, and set these coefficients as our $p_1^I$, $p_2^I$, and $p_3^I$.

By expanding \cref{eq:naive-f}, we obtain
\begin{equation}\label{eq:f-expand}
f(t)f(t+k_1d)f(t+k_2d)f(t+k_3d)
= \alpha^4 \sum_{-4 \le i_1, i_2, i_3, i_4 \le 4} \gamma_{i_1}\gamma_{i_2}\gamma_{i_3}\gamma_{i_4} u_{i_1, i_2, i_3, i_4} (t).
\end{equation}

There are always three ``main terms'' (c.f. \cref{prop:0125}) coming from the signatures $(0,0,0,0)$ and $\pm(1,2,3,4)$.

\begin{proposition}\label{prop:triviality-of-main-terms}
$u_{0,0,0,0}=u_{1,2,3,4}=u_{-1,-2,-3,-4}=1$.
\end{proposition}

\begin{proof}
This follows from \cref{eq:a-quadratic-identity}.
\end{proof}

Since we ultimately care about fixing some nonzero value of $d$ and averaging over $t$, we are primarily concerned with cases in which $p_1^I,p_2^I$ both vanish at the point $(x,y,z)$ corresponding to our pattern $0,k_1,k_2,k_3$. This leads to a natural notion of degeneracy.
\begin{definition}
A signature $I$ is \emph{degenerate at pattern $P$} if $P = \{0,x,x+y,2x+y+z\}$ and
\[p_1^I(x,y,z) = p_2^I(x,y,z) = 0;\]
otherwise it is \emph{nondegenerate} at $P$.
\end{definition}

\begin{remark}
The signatures $(0,0,0,0)$ and $\pm (1,2,3,4)$ are always degenerate by \cref{prop:triviality-of-main-terms}.
\end{remark}
In particular, we have the following estimate.
\begin{lemma}\label{lem:gauss-sum}
If $I$ is nondegenerate at pattern $P= \{0,x,x+y,2x+y+z\}$, then for all nonzero $d \in \ZZ/N\ZZ$, 
\[
\abs{\EE_t u_I(t,d)} \le N^{-1/2}.
\]
\end{lemma}

\begin{proof}
Since $I$ is nondegenerate, one has $u_I(t,d) = \omega^{at^2 + btd + cd^2}$ where $a$ and $b$ are not both zero in $\ZZ/N\ZZ$. The claim follows by the standard Gauss sum estimate (recall that $N$ is prime).
\end{proof}

Now we characterize which signatures $I$ can be degenerate at a pattern, and for those signatures, which patterns they will be degenerate at. Here we write $\NN \cdot (a,b,c) = \{ (na,nb,nc) : n \in \NN\}$.

\begin{definition}
Given a pattern $P = \{0,x,x+y,2x+y+z\}$ corresponding to the triple $(x,y,z) \in \NN^3$, let $\mc{I}(P)$ be the set of signatures $I$ which are degenerate at $P$. We call $\mc{I}(P)$ the \emph{degeneracy set} of $P$.
\end{definition}

\begin{lemma}\label{lem:degeneracy-bash}
Let $S = \{(0,0,0,0),\pm(1,2,3,4)\}$. Let $x,y,z \in \NN$ and $P= \{0,x,x+y,2x+y+z\}$. Then
\[
\mc{I}(P)\setminus S =
\begin{cases}
\{\pm(1,-3,1,0),\pm(1,0,-3,1),\pm(3,-3,-1,1)\} &\mbox{if } (x,y,z)\in \NN \cdot(1,1,1), \\
\{\pm(0,3,2,3)\} & \mbox{if } (x,y,z)\in \NN \cdot(1,3,2),\\
\{\pm(3,0,-1,3)\} & \mbox{if } (x,y,z)\in \NN \cdot(1,4,4),\\
\{\pm(4,0,1,4)\} & \mbox{if } (x,y,z)\in \NN \cdot(2,1,1),\\
\{\pm(1,2,1,4),\pm(3,2,1,4),\pm(3,2,3,4)\} & \mbox{if } 2x^2+xz-yz=0,\\
\{\pm(3,-3,-1,1)\} & \mbox{if } x^2-yz=0 \mbox{ and } (x,y,z)\neq\ell\cdot(1,1,1),\\
\{\pm(2,3,-2,-3)\} & \mbox{if } x^2 + xz - y^2=0,\\
\{\pm(2,1,-2,-1)\} & \mbox{if } 2x^3 + 4x^2y + x^2z + 2xy^2 - y^2z - yz^2=0,\\
\{\pm(3,-1,1,-3)\} & \mbox{if } 4x^3 + 4x^2y + 4x^2z + 2xyz + xz^2 - y^2z=0,\\
\emptyset & \mbox{otherwise.}
\end{cases}
\]
Furthermore if $2x^2+xz-yz=0$, then 
\[p_3^{\pm(1,2,1,4)}(x,y,z) = p_3^{\pm(3,2,1,4)}(x,y,z) = p_3^{\pm(3,2,3,4)}(x,y,z) =0.\]  
\end{lemma}

See \cref{app:appendix-proof} for a computer-assisted proof of \cref{lem:degeneracy-bash}.

\begin{proof}[Proof of \cref{thm:1d-general}]
Recall the definition of $f$ from \cref{eq:naive-f}, which depended on some yet-to-be-chosen real constants $\gamma_i$.
Recalling the convention that $\gamma_{0}=1$ and $\gamma_{-k} = \gamma_k$ for $1 \le k \le 4$.
For a signature $I = (i_1, i_2,i_3,i_4)$, write 
$\gamma_I = \gamma_{i_1}\gamma_{i_2}\gamma_{i_3}\gamma_{i_4}$.

Let $S = \{(0,0,0,0),\pm(1,2,3,4)\}$.
Using the expansion \cref{eq:f-expand} and the Gauss sum estimate \cref{lem:gauss-sum}, we have 
\begin{align}
&\hspace{-4em}\max_{0 \ne d \in \ZZ/N\ZZ} \EE_t[f(t)f(t+k_1d)f(t+k_2d)f(t+k_3d)] \notag 
\\&=\alpha^{4}\sum_{I\in\mc{I}(P)} \gamma_I \omega^{p_3^{I}(x,y,z)d^2} + O(N^{-1/2}) \notag
\\
&= \alpha^{4}\bigg(1 + 2\gamma_1\gamma_2\gamma_3\gamma_4 + \sum_{I \in \mc{I}(P)\setminus S} \gamma_I \omega^{p_3^{I}(x,y,z)d^2}\bigg) + O(N^{-1/2})\label{eq:calc-intermediate}
\\ 
&\le \alpha^4\bigg(1+2\gamma_1\gamma_2\gamma_3\gamma_4 +\sum_{I\in\mc{I}(P)\setminus S} |\gamma_I|\bigg)+O(N^{-1/2}).\label{eq:calc-main}
\end{align}
The remainder of the proof splits into the cases of \cref{lem:degeneracy-bash} depending on $\mc{I}\setminus S$.
In all cases other than the fifth case, we will show that it is possible to choose constants $\gamma_1, \dots, \gamma_4$ so that
\begin{equation}\label{eq:4-point-finish}
1+2\gamma_1\gamma_2\gamma_3\gamma_4 +\sum_{I\in\mc{I}(P)\setminus S} |\gamma_I| < 1,
\end{equation}
which would imply the claimed inequality. As an example, we explicitly work out the first case of \cref{lem:degeneracy-bash}, namely when $(x,y,z)\in \NN \cdot(1,1,1)$. By \cref{lem:degeneracy-bash},
\[\mc{I}(P)\setminus S=\{\pm(1,-3,1,0),\pm(1,0,-3,1),\pm(3,-3,-1,1)\}.\]
Plugging into \eqref{eq:calc-main}, we obtain
\[\EE[f(t)f(t+k_1d)f(t+k_2d)f(t+k_3d)]\le\alpha^2(1+2\gamma_1\gamma_2\gamma_3\gamma_4 + 4\gamma_1^2|\gamma_3| + 2\gamma_1^2\gamma_3^2) + O(N^{-1/2}).\] Choosing $\gamma_1 = -1/512$ and $\gamma_2 = \gamma_3 = \gamma_4 = 1/8$ we have
\[1+2\gamma_1\gamma_2\gamma_3\gamma_4 + 4\gamma_1^2|\gamma_3| + 2\gamma_1^2\gamma_3^2<1,\]
which establishes \cref{eq:4-point-finish} in this case. Note that this discussion matches that in the paragraph following \cref{prop:0125}.

For the sixth, eighth, ninth, and tenth cases in \cref{lem:degeneracy-bash} setting $\gamma_1 = -1/512$ and $\gamma_2 = \gamma_3 = \gamma_4 = 1/8$ establishes \cref{eq:4-point-finish} in an analogous fashion.

For the second, third, and seventh cases, we set $\gamma_3 = -1/512$ and $\gamma_1 = \gamma_2 = \gamma_4 = 1/8$ in order to establish \cref{eq:4-point-finish}.

For the fourth case, we set $\gamma_4 = -1/512$ and $\gamma_1 = \gamma_2 = \gamma_3 = 1/8$ in order to establish \cref{eq:4-point-finish}.

Finally, for the fifth case in \cref{lem:degeneracy-bash}, the above bounding is too crude and we must use the extra information from \cref{lem:degeneracy-bash} that $p_3^{\pm(1,2,1,4)}(x,y,z) = p_3^{\pm(3,2,1,4)}(x,y,z) = p_3^{\pm(3,2,3,4)}(x,y,z) = 0$ in this case. Using \eqref{eq:calc-intermediate} and using $p_3^I(x,y,z)=0$ for these values $I\in\mc{I}(P)\setminus S$, we find that 
\[\EE[f(t)f(t+k_1d)f(t+k_2d)f(t+k_3d)] = \alpha^4(1 + 2(\gamma_1+\gamma_3)^2\gamma_2\gamma_4) + O(N^{-1/2})).\]
Then taking $\gamma_2 = -1/8$ and $\gamma_1 = \gamma_3 = \gamma_4 = 1/8$ suffices since then
\[1 + 2(\gamma_1+\gamma_3)^2\gamma_2\gamma_4<1.\qedhere\]
\end{proof}

%%% AUTHOR: optional appendix here
\appendix %% you may comment this out if no Appendix
\section*{Appendix}
\section{Characterizing degeneracy sets}\label{app:appendix-proof}
The aim of this appendix is to prove \cref{lem:degeneracy-bash}. The computer code (in Python and Magma) are included as ancillary files in the arXiv version of this paper.

\begin{lemma}\label{lem:high-deg}
Let $T$ be the set of curves 
\[\{2X^3 + 2X^2Y + 3X^2Z + XYZ + XZ^2 - Y^2Z,\]
\[ 2X^3 + 2X^2Y + X^2Z - XYZ - Y^2Z - YZ^2,\]
\[2X^4 + 2X^3Y + 3X^3Z - X^2YZ + X^2Z^2 - 4XY^2Z - 3XYZ^2 - Y^3Z - 2Y^2Z^2 - YZ^3,\]
\[2X^4 + 2X^3Y + 5X^3Z + 3X^2YZ + 4X^2Z^2 - 2XY^2Z + XYZ^2 + XZ^3 - Y^3Z - Y^2Z^2,\]
\[2X^4 + 4X^3Y + 3X^3Z + 2X^2Y^2 + 2X^2YZ + X^2Z^2 - 2XYZ^2 - Y^2Z^2 - YZ^3,\]
\[2X^4 + 4X^3Y + 5X^3Z + 2X^2Y^2 + 5X^2YZ + 4X^2Z^2 - XY^2Z + 2XYZ^2 + XZ^3 - Y^3Z - Y^2Z^2\}.\]
None of the curves in $T$ have a positive rational solution.
\end{lemma}
\begin{proof}
This is proved using a variety of computational tools in Magma. To briefly outline the approach, the first two curves are genus $1$ and Magma first proves that the curves have rank $0$. Then the size of the torsion subgroup is computed and searching over points of small height the associated points are found. One checks that none of these correspond to positive rational solutions.

The remaining four curves are genus $2$ and we used Magma to compute the rank of the associated Jacobian to be $0$. 
Magma then computed all the rational points on these curves and verified that none of them are positive rational solutions. This is done using a variant of the Chabauty method (see \cite{McPo12} for further details on such methods).
\end{proof}
We now reduce \cref{lem:degeneracy-bash} to a set of modular claims which are then verified with computer assistance. For each of the $9^4 = 6561$ signatures $I$ we compute $p_1^{I}(X,Y,Z)$ and $p_2^{I}(X,Y,Z)$. Our analysis then proceeds into three separate cases based on whether $p_1^{I}(X,Y,Z)$ and $p_2^{I}(X,Y,Z)$ vanish as polynomials in $X,Y,Z$.

\begin{definition}\label{def:a2}
Given a signature $I$, define its \emph{pattern set} to be
\[
\mc{Q}(I) = \{(x,y,z) \in \NN^3 : p_1^I(x,y,z) = 0 \text{ and } p_2^I(x,y,z) = 0\}.
\]
\end{definition}

\begin{claim}[Signatures with $p_1^I \equiv 0$ and $p_2^I \equiv 0$]\label{claim:a3}
The pattern $I$ satisfies $p_1^{I}(X,Y,Z) = p_2^{I}(X,Y,Z) = 0$ as polynomials (equivalently $\mc{Q}(I) = \NN^3$) if and only if $I\in\{(0,0,0,0),\pm(1,2,3,4)\}$.
\end{claim}

\begin{proof}
Verified with computer assistance by a brute-force search over signatures.
\end{proof}

The signatures $\{(0,0,0,0),\pm(1,2,3,4)\}$ occur in the degeneracy set of every pattern.

\begin{claim}[Signatures with $p_1^I \equiv 0$ and $p_2^I \not\equiv 0$]\label{claim:a4}
Let $I$ be a signature.
\begin{enumerate}[(a)]
    \item If $I=\pm(3,2,1,4)$ then $\mc{Q}(I) = \{(x,y,z)\in\NN^3: 2x^2+xz-yz=0\}$.
    \item If $I=\pm(3,-3,-1,1)$ then $\mc{Q}(I) = \{(x,y,z)\in\NN^3: x^2-yz=0\}$.
    \item If $I=\pm(2,3,-2,-3)$ then $\mc{Q}(I) = \{(x,y,z)\in\NN^3: x^2+xz-y^2=0\}$.
    \item If $I=\pm(2,1,-2,-1)$ then $\mc{Q}(I) = \{(x,y,z)\in\NN^3: 2x^3+4x^2y+x^2z+2xy^2-y^2z-yz^2=0\}$.
    \item If $I=\pm(3,-1,1,-3)$ then $\mc{Q}(I) = \{(x,y,z)\in\NN^3: 4x^3+4x^2y+4x^2z+2xyz+xz^2-y^2z=0\}$.
    \item If $I$ is a signature such that $p_1^I(X,Y,Z) = 0$ and $p_2^I(X,Y,Z)\neq 0$ and $I$ is not one of the above signatures, then $\mc{Q}(I)= \emptyset$.
\end{enumerate}
\end{claim}
\begin{proof}
For each $I$ in the first five cases we can check that $p_1^I(X,Y,Z) = 0$. We then compute a factorization of $p_2^I(X,Y,Z)$ and remove all factors with all positive coefficients (which can never vanish for $x,y,z > 0$).
The remaining factor is recorded above.

For each $I$ in the final case, in which $p_1^I(X,Y,Z) = 0$ and $p_2^I(X,Y,Z)\neq 0$, with computer assistance, we verify that that one of the following is true:
\begin{itemize}
    \item $(2X+Y+Z)^2p_2^I(X,Y,Z)$ has all coefficients of the same sign;
    \item all the factors of $p_2^I(X,Y,Z)$ lie in $T \cup \{X,Y,Z,Y+Z,X+Z,X+Y,X+Y+Z,2X+Y+Z,2X+2Y+Z\}$ ($T$ was defined in \cref{lem:high-deg}).
\end{itemize}
In both cases, we see that $\mc{Q}(I) = \emptyset$.
\end{proof}
\begin{claim}[Signatures with $p_1^I \not\equiv 0$ and $p_2^I \equiv 0$]\label{claim:a5}
Let $I$ be a signature.
\begin{enumerate}[(a)]
    \item If $I = \pm(3,2,3,4)$ then $\mc{Q}(I) = \{(x,y,z)\in\NN^3: 2x^2+xz-yz=0\}$.
    \item If $I$ is a signature such that $p_1^I(X,Y,Z) \neq 0$ and $p_2^I(X,Y,Z)= 0$ and $I\neq \pm(3,2,3,4)$, then $\mc{Q}(I)= \emptyset$.
\end{enumerate}
\end{claim}
\begin{proof}
For $I = \pm(3,2,3,4)$ we can check that $p_2^I(X,Y,Z) = 0$. We then compute a factorization of $p_1^I(X,Y,Z)$ and remove all factors with all positive coefficients.
The remaining factor is recorded above.

In the final case, we check that $(2X+Y+Z)^2p_1^I(X,Y,Z)$ 
has all coefficients of the same sign, from which we deduce $\mc{Q}(I) = \emptyset$.
\end{proof}
\begin{claim}\label{claim:a6}
Let $\mc{I}_0$ be the set of signatures $I$ such that following property holds: the polynomials 
\begin{itemize}
    \item $f_1(X,Y,Z) = (2X+Y+Z)^4p_1^I(X,Y,Z)$
    \item $f_2(X,Y,Z) = (2X+Y+Z)^4p_2^I(X,Y,Z)$
    \item $f_3(X,Y,Z) = (X+Y)f_1(X,Y,Z)-f_2(X,Y,Z)$ 
    \item $f_4(X,Y,Z)=(2X+Y+Z)f_1(X,Y,Z)-f_2(X,Y,Z)$
    \item  $f_5(X,Y,Z)=Xf_1(X,Y,Z)-f_2(X,Y,Z)$
\end{itemize}
are all nonzero and each does not have all of its coefficients the same sign. Then $\abs{\mc{I}_0} = 122$.
\end{claim}
\begin{proof}
Verified with computer assistance by a brute-force search over signatures.
\end{proof}

\begin{claim}[Signatures with $p_1^I \not\equiv 0$ and $p_2^I \not\equiv 0$]
\label{claim:a7}
Let $I$ be a signature.
\begin{enumerate}[(a)]
    \item If $I=\pm(1,2,1,4)$ then $\mc{Q}(I) = \{(x,y,z)\in\NN^3: 2x^2+xz-yz=0\}$.
    \item If $I=\pm(1,-3,1,0)$ then $\mc{Q}(I) = \{\ell\cdot(1,1,1): \ell\in\NN\}$.
    \item If $I=\pm(1,0,-3,1)$ then $\mc{Q}(I) = \{\ell\cdot(1,1,1): \ell\in\NN\}$.
    \item If $I=\pm(0,3,2,3)$ then $\mc{Q}(I) = \{\ell\cdot(1,3,2): \ell\in\NN\}$.
    \item If $I=\pm(3,0,-1,3)$ then $\mc{Q}(I) = \{\ell\cdot(1,4,4): \ell\in\NN\}$.
    \item If $I=\pm(4,0,1,4)$ then $\mc{Q}(I) = \{\ell\cdot(2,1,1): \ell\in\NN\}$.
    \item If $I$ is a signature such that $p_1^I(X,Y,Z) \neq 0$ and $p_2^I(X,Y,Z)\neq 0$ and $I$ is not one of the above signatures, then $\mc{Q}(I)= \emptyset$.
\end{enumerate}
\end{claim}

\begin{proof}
For any signature $I\notin\mc{I}_0$ 
(as defined in \cref{claim:a6}) with
$p_1^I(X,Y,Z)\neq 0$ and $p_2^I(X,Y,Z)\neq 0$ (as polynomials), one has
$\mc{Q}(I) = \emptyset$ by an easy application of \cref{claim:a6}.  This is because, e.g., such a signature will satisfy a property such as
\[(X+Y)f_1(X,Y,Z)-f_2(X,Y,Z)\]
has all positive coefficients and is a nonzero polynomial, where $f_j(X,Y,Z) = (2X+Y+Z)^4p_j^I(X,Y,Z)$ for $j=1,2$. But any $(x,y,z)\in\NN^3$ satisfying $p_1^I(x,y,z) = p_2^I(x,y,z) = 0$ must be a root of this polynomial, which is a contradiction as $x,y,z$ are positive.

For the remainder of the proof, we can assume that $I \in \mc I_0$.

For $I=\pm(1,2,1,4)$, we compute
\[\gcd(p_1^I(X,Y,Z),p_2^I(X,Y,Z)) = (X+Y+Z)(2X^2+XZ-YZ),\]
hence the result. 

For each $I\in\mc{I}_0\setminus\{\pm(1,2,1,4)\}$, which is a total of $120$ explicit cases, using Magma, we check that the equations $p_1^I(X,Y,Z)=0$ and $p_2^I(X,Y,Z)=0$ cut out a zero-dimensional subscheme of the projective space $\PP^2_{\QQ}$ 
(in Magma, such objects are called ``clusters''). Using the \texttt{RationalPoints} function in Magma, we compute all rational ratios $X:Y:Z$ that provide a solution to both equations. This function is rigorous when applied to zero-dimensional schemes. We record the positive solutions as $\mc{Q}(I)$.
\end{proof}

\begin{table}[t]
    \centering
    \begin{tabular}{cc}
    $I$ & $\mc{Q}(I)$\\
    \midrule
     $\pm(3,2,1,4)$   & $\{2x^2+xz-yz=0\}$ \\
     $\pm(3,-3,-1,1)$ & $\{x^2-yz=0\}$ \\
     $\pm(2,3,-2,-3)$ & $\{x^2+xz-y^2=0\}$\\
     $\pm(2,1,-2,-1)$ & $\{2x^3+4x^2y+x^2z+2xy^2-y^2z-yz^2=0\}$\\
     $\pm(3,-1,1,-3)$ & $\{4x^3+4x^2y+4x^2z+2xyz+xz^2-y^2z=0\}$\\
     \midrule
     $\pm(3,2,3,4)$ & $\{2x^2+xz-yz=0\}$\\
     \midrule
    $\pm(1,2,1,4)$ &  $\{2x^2+xz-yz=0\}$\\
    $\pm(1,-3,1,0)$ & $\NN\cdot(1,1,1)$\\
    $\pm(1,0,-3,1)$ & $\NN\cdot(1,1,1)$\\
    $\pm(0,3,2,3)$ & $\NN\cdot(1,3,2)$\\
    $\pm(3,0,-1,3)$ & $\NN\cdot(1,4,4)$\\
    $\pm(4,0,1,4)$ & $\NN\cdot(2,1,1)$\\
    \midrule
    $(0,0,0,0)$ & $\NN^3$ \\
    $\pm(1,2,3,4)$ & $\NN^3$ \\
    \midrule 
    Otherwise & $\emptyset$ \\
    \end{tabular}
    \caption{The pattern set for each signature}
    \label{tab:signature-pattern}
\end{table}

Claims \ref{claim:a3}, \ref{claim:a4}, \ref{claim:a5}, and \ref{claim:a7}
together cover all signatures. 
They are summarized in \cref{tab:signature-pattern}.

Now, in order to compute the possible degeneracy sets $\mc{I}(P)$ for a pattern $P$, note that $P=\{0,x,x+y,2x+y+z\}$ must satisfy $p_1^I(x,y,z) = p_2^I(x,y,z) = 0$ for any $I\in\mc{I}(P)$ and therefore $P\in \mc{Q}(I)$. Therefore any such $I\in\mc{I}(P)$ must have a nonempty pattern set so must be one of the signatures explicitly listed in \cref{tab:signature-pattern}.

Note that any $I\in S = \{(0,0,0,0),\pm(1,2,3,4)\}$ will be in every $\mc{I}(P)$ by \cref{claim:a3}. Beyond that $\mc{I}(P)$ is composed of the signatures $I$ listed above which contain $(x,y,z)$ in their pattern set. Therefore it remains merely to understand how the pattern sets in \cref{tab:signature-pattern} divide up the space of patterns.

\begin{claim}\label{claim:a8}
Let $I_1$ and $I_2$ be signatures in
\[\{\pm(3,2,1,4), \pm(3,-3,-1,1), \pm(2,3,-2,-3),\pm(2,1,-2,-1), \pm(3,-1,1,-3), \pm(3,2,3,4), \pm(1,2,1,4)\}.\]
Then either $\mc{Q}(I_1) = \mc{Q}(I_2)$ or $\mc{Q}(I_1) \cap \mc{Q}(I_2)=\emptyset$. 
\end{claim}
\begin{proof}
It suffices to verify that no pair of equations, e.g. $2x^2+xz-yz=0$ and $2x^3+4x^2y+x^2z+2xy^2-y^2z-yz^2=0$, has a positive rational solution. This is done by verifying in Magma that they cut out a zero-dimensional scheme in $\PP^2_{\QQ}$ and then using \texttt{RationalPoints} to verify that there is no positive rational solution.
\end{proof}

Chaining all these claims, we finally deduce \cref{lem:degeneracy-bash}.

%%% AUTHOR: optional acknowledgments here
\section*{Acknowledgments} %%  you may comment this out if no Ackno
The authors are grateful to the anonymous reviewers for several suggestions which helped improve the presentation of the paper.

%%% AUTHOR:
%%% Bibliography goes here. Note that the arXiv cannot process bibtex
%%% or biber bibliographies.  Example of acceptable bibliograpy format:

%% AUTHOR: You can generate such a bibliography from a .bib file by 
%% running pdflatex/bibtex/pdflatex/pdflatex and then pasting the .bbl file
%% between \begin{thebibliography} and \end{bibliography}

%%% AUTHOR: Include a short description of each author following the
%%% structure below. Use the same short tags used previously.  
%%% Use \imageat{} and \imagedot{} instead of "@" and "." in
%%% email addresses-this replaces the symbols with graphics to avoid 
%%% e-mail address harvesting from the .pdf file
\begin{dajauthors}
\begin{authorinfo}[Sah]
  Ashwin Sah\\
  Massachusetts Institute of Technology\\
  Cambridge, MA\\
  asah\imageat{}mit\imagedot{}edu \\
  \url{http://www.mit.edu/~asah}
\end{authorinfo}
\begin{authorinfo}[Sawhney]
  Mehtaab Sawhney\\
  Massachusetts Institute of Technology\\
  Cambridge, MA\\
  msawhney\imageat{}mit\imagedot{}edu \\
  \url{http://www.mit.edu/~msawhney}
\end{authorinfo}
\begin{authorinfo}[Zhao]
  Yufei Zhao\\
  Massachusetts Institute of Technology\\
  Cambridge, MA\\
  yufeiz\imageat{}mit\imagedot{}edu \\
  \url{https://yufeizhao.com}
\end{authorinfo}
\end{dajauthors}

\end{document}